\documentclass[10pt]{amsart}

\usepackage{amssymb}
\usepackage{amsthm}
\usepackage{amsmath}

\usepackage[usenames]{color}
\definecolor{ANDREW}{RGB}{255,127,0}

\usepackage{tikz}
\usetikzlibrary{arrows,calc}

\usepackage[foot]{amsaddr}

\usepackage{hyperref}

\theoremstyle{plain}
\newtheorem{proposition}{Proposition}[section]
\newtheorem{theorem}[proposition]{Theorem}
\newtheorem{lemma}[proposition]{Lemma}
\newtheorem{corollary}[proposition]{Corollary}
\theoremstyle{definition}

\newtheorem{definition}[proposition]{Definition}
\newtheorem{observation}[proposition]{Observation}
\theoremstyle{remark}
\newtheorem{remark}[proposition]{Remark}
\newtheorem{conjecture}[proposition]{Conjecture}
\newtheorem*{question}{Question}

\DeclareMathOperator{\Aut}{Aut}

\DeclareMathOperator{\Vol}{Vol}

\DeclareMathOperator{\Hol}{Hol}

\DeclareMathOperator{\Euc}{Euc} 
\DeclareMathOperator{\id}{id}

\DeclareMathOperator{\Cc}{\mathcal{C}}

\DeclareMathOperator{\Tc}{\mathcal{T}}

\DeclareMathOperator{\Bb}{\mathbb{B}}
\DeclareMathOperator{\Cb}{\mathbb{C}}
\DeclareMathOperator{\Db}{\mathbb{D}}

\DeclareMathOperator{\Nb}{\mathbb{N}}

\DeclareMathOperator{\Rb}{\mathbb{R}}

\newcommand{\abs}[1]{\left|#1\right|}

\newcommand{\norm}[1]{\left\|#1\right\|}
\newcommand{\wt}[1]{\widetilde{#1}}
\newcommand{\wh}[1]{\widehat{#1}}
\newcommand{\ip}[1]{\left\langle #1\right\rangle}

\begin{document}

\title[Boundary Rigidity Results]{Two boundary rigidity results for holomorphic maps}
\author{Andrew Zimmer}\address{Department of Mathematics, Louisiana State University, Baton Rouge, LA, USA}
\email{amzimmer@lsu.edu}
\date{\today}
\keywords{Holomorphic mapping, boundary rigidity, Schwarz lemma, Kobayashi metric, K{\"a}hler geometry, bounded geometry}
\subjclass[2010]{32H12, 	32F45, 32A40, 	53B35, 53C22}

\begin{abstract} In this paper we establish two boundary versions of the Schwarz lemma. The first is for general holomorphic self maps of bounded convex domains with $C^2$ boundary. This appears to be the first boundary Schwarz lemma for general holomorphic self maps that requires no strong pseudoconvexity or finite type assumptions. The second is for biholomorphisms of domains who have an invariant K{\"a}hler metric with bounded sectional curvature. This second result applies to holomorphic homogeneous regular domains and appears to be the first boundary Schwarz lemma that makes no assumptions on the regularity of the boundary. 
\end{abstract}

\maketitle

\section{Introduction}

In 1931 Cartan proved the following generalization of the Schwarz lemma.

\begin{theorem}[Cartan's Uniqueness Theorem]\label{thm:Cartan} If $\Omega \subset \Cb^d$ is a bounded domain, $f : \Omega \rightarrow \Omega$ is a holomorphic map, and  there exists $z_0 \in \Omega$ such that 
\begin{align*}
f(z) = z + { \rm o}\left( \norm{z-z_0}\right),
\end{align*}
then $f = \id$. 
\end{theorem}

It seems natural to ask if a similar result holds when $z_0 \in \partial \Omega$. In this case the problem is much harder and  already in the very special case of biholomorphisms of the unit disk a higher order error term is necessary for rigidity.  

In this paper we prove two new boundary versions of Theorem~\ref{thm:Cartan}. Our first main result, see Theorem~\ref{thm:main_convex} below, extends a well known theorem of Burns and Krantz to any bounded convex domain with $C^2$ boundary (assuming a slightly worse error term). This appears to be the first boundary Schwarz lemma for general holomorphic self maps that requires no strong pseudoconvexity or finite type assumptions.  Our second main result, see Theorem~\ref{thm:main_intro_biholo} below, establishes a boundary Schwarz lemma for biholomorphisms of domains which have an invariant K{\"a}hler metric with certain bounded geometry properties. This applies to  holomorphic homogeneous regular domains and appears to be the first boundary Schwarz lemma that makes no assumptions on the regularity of the boundary. 

\subsection{General holomorphic self maps}

The first boundary Schwarz lemma for general holomorphic self maps is due to Burns and Krantz who established the following.

\begin{theorem}[Burns-Krantz~\cite{BK1994}] \label{thm:BK} Suppose $\Omega \subset \Cb^d$ is a bounded strongly pseudoconvex domain with $C^6$ boundary. If $f : \Omega \rightarrow \Omega$ is a holomorphic map and there exists $\xi_0 \in \partial \Omega$ such that 
\begin{align*}
f(z) = z + { \rm o}\left( \norm{z-\xi_0}^3\right),
\end{align*}
then $f = \id$. 
\end{theorem}

As observed by Burns and Krantz, the error term in Theorem~\ref{thm:BK} is already optimal when $\Omega$ is the unit disk (see Remark 1 in~\cite{BK1994}). 

A number of similar boundary Schwarz lemmas for holomorphic self maps have been established, see for instance~\cite{O2000, C2001, BZZ2006, B2008, LT2016, TLZ2017} and the survey article~\cite{K2011}. However most of these results either assume  that $d=1$ or that the domain is strongly pseudoconvex. For weakly pseudoconvex domains, the following conjecture has been attributed to Burns and Krantz (see~\cite[pg. 312]{H1993}).

\begin{conjecture}[Burns-Krantz] Let $\Omega \subset \Cb^d$ be a pseudoconvex domain of finite type and suppose that $\xi_0 \in \partial \Omega$. Then there exists some $m$ which depends on the geometry of $\partial \Omega$ at $\xi_0$ such that: if $f : \Omega \rightarrow \Omega$ is a holomorphic map and \begin{align*}
f(z) = z + { \rm o}\left( \norm{z-\xi_0}^{m}\right),
\end{align*}
then $f = \id$. 
\end{conjecture}

Huang gave a positive answer to the above conjecture for convex domains of finite type. In his result the error term depends on the line type, denoted by $\ell(\xi_0)$, of the boundary point $\xi_0 \in \partial \Omega$ (see Section~\ref{sec:finite_type} for the definition). More precisely:

\begin{theorem}[{Huang~\cite[Theorem 0.4]{Huang1995}}]\label{thm:huang} Suppose that $\Omega \subset \Cb^d$ is a bounded convex domain of finite type. If $f : \Omega \rightarrow \Omega$ is a holomorphic map and there exists $\xi_0 \in \partial \Omega$ such that 
\begin{align*}
f(z) = z + { \rm o}\left( \norm{z-\xi_0}^{m}\right)
\end{align*}
for some $m > 5 \ell(\xi_0)$, then $f = \id$. 
\end{theorem}

Despite the high order error term in Huang's result, to the best of our knowledge there is no example of a smoothly bounded pseudoconvex domain $\Omega \subset \Cb^d$ with a holomorphic map $f : \Omega \rightarrow \Omega$ and a boundary point $\xi_0 \in \partial \Omega$ such that $f \neq \id$ and
\begin{align*}
f(z) = z + { \rm o}\left( \norm{z-\xi_0}^{m}\right)
\end{align*}
for some $m > 3$. So exactly how finite type relates to the existence of boundary Schwarz lemmas and the optimal error term is completely mysterious. 

In the first main theorem of this paper, we establish a boundary Schwarz lemma for convex domains which sheds some light on this mystery and in particular shows that when the domain is convex, finite type conditions are not necessary. 

\begin{theorem}\label{thm:main_convex} (see Section~\ref{sec:pf_of_thm_main}) Suppose $\Omega \subset \Cb^d$ is a bounded convex domain with $C^2$ boundary. If $f : \Omega \rightarrow \Omega$ is a holomorphic map and there exists $\xi_0 \in \partial \Omega$ such that 
\begin{align*}
f(z) = z + { \rm o}\left( \norm{z-\xi_0}^4\right),
\end{align*}
then $f = \id$. 
\end{theorem}

\begin{remark} \ 
\begin{enumerate}
\item Theorem~\ref{thm:main_convex} is new even in the very special case when $f$ is a biholomorphism, $\partial \Omega$ is $C^\infty$, and $d=2$.
\item To the best of our knowledge, Theorem~\ref{thm:main_convex} is the only known boundary Schwarz lemma for general holomorphic self maps that makes no strong pseudoconvexity or finite type assumptions. 
\item  It is unclear whether $\norm{z-\xi_0}^4$ can be improved to $\norm{z-\xi_0}^3$. \end{enumerate}
\end{remark}

In the case when $\partial \Omega$ is smooth and $\xi_0 \in \partial \Omega$ has finite line type we can give a slight improvement to the error term.

\begin{theorem}\label{thm:main_finite_type}(see Section~\ref{sec:finite_type}) Suppose $\Omega \subset \Cb^d$ is a bounded convex domain with $C^\infty$ boundary and $f : \Omega \rightarrow \Omega$ is a holomorphic map. If there exists $\xi_0 \in \partial \Omega$ such that $\ell(\xi_0) < +\infty$ and 
\begin{align*}
f(z) = z + { \rm o}\left( \norm{z-\xi_0}^{4-1/\ell(\xi_0)}\right),
\end{align*}
then $f = \id$. 
\end{theorem}

Motivated by Theorem~\ref{thm:main_convex} we make the following conjecture. 

\begin{conjecture}\label{conj:convex} Suppose $\Omega \subset \Cb^d$ is a bounded convex domain and $\xi_0 \in \partial \Omega$. Then there exists $m=m(\xi_0)$ which only depends on the tangent cone of $\overline{\Omega}$ at $\xi_0$ such that: if $f : \Omega \rightarrow \Omega$ is a holomorphic map and 
\begin{align*}
f(z) = z + { \rm o}\left( \norm{z-\xi_0}^{m}\right),
\end{align*}
then $f = \id$. 
\end{conjecture}

In the $d=1$ case the conjecture follows from the Burns-Krantz theorem for the unit disk, the Riemann mapping theorem, and  estimates on the Kobayashi distance. In Corollary~\ref{cor:convex1} below, we show that the conjecture is true in the special case when $f$ is a biholomorphism. 

\subsection{The special case of biholomorphisms} For a bounded domain $\Omega \subset \Cb^d$ let $\Aut(\Omega)$ denote the automorphism group of $\Omega$, that is the group of biholomorphic maps $\Omega \rightarrow \Omega$. 

In the special case when $f:\Omega \rightarrow \Omega$ is a biholomorphism, there are many extensions of the Burns-Krantz theorem, see for instance~\cite{BK1994, BER2000, ELZ2003,LM2007,LM2007b,J2009,BC2014}. Many of these results are in the setting of CR-manifolds and so to apply them to bounded domains, one first needs to show that the biholomorphism extends to a CR-automorphism of the boundary and then use the CR-geometry of the boundary to obtain a rigidity result. 

For instance, Bell and Ligocka~\cite{BL1980} proved that if $\Omega \subset \Cb^d$ is a bounded pseudoconvex domain with real analytic boundary, then every $\varphi\in \Aut(\Omega)$ extends to a CR-automorphism $\partial \Omega \rightarrow \partial \Omega$. Then, using the CR-geometry of the boundary, Baouendi, Ebenfelt, and Rothschild proved the following.

\begin{theorem}[{Baouendi-Ebenfelt-Rothschild~\cite[Theorem 5]{BER2000}}]\label{thm:real_analytic} Suppose $\Omega \subset \Cb^d$ is a bounded pseudoconvex domain with real analytic boundary and $\xi_0 \in \partial\Omega$. Then there exists $L=L(\xi_0)>0$ such that: if $\varphi \in \Aut(\Omega)$ and 
\begin{align*}
\varphi(z) = z + { \rm O}\left(\norm{z-\xi_0}^L\right),
\end{align*}
then $\varphi = \id$.
\end{theorem} 

\begin{remark} With the hypothesis of Theorem~\ref{thm:real_analytic}, Lamel and Mir~\cite[Corollary 1.4]{LM2007} proved that $L$ can be chosen to depend only on $\partial \Omega$.
\end{remark}

In the second main theorem of this paper, we establish an alternative approach to these types of results which makes no assumptions about the CR-geometry of the boundary and instead only makes assumptions about the intrinsic complex geometry of the domain. In particular, we will assume that there exists an invariant K{\"a}hler metric  with certain bounded geometry properties.

Given a domain $\Omega \subset \Cb^d$ and $z \in \Omega$ define 
\begin{align*}
\delta_\Omega(z) = \inf\{ \norm{w-z} : w \in \partial \Omega \}.
\end{align*}

\begin{definition} Suppose $\Omega \subset \Cb^d$ is a bounded domain. A complete K{\"a}hler metric $g$ on $\Omega$ has \emph{property-$(BG)$} if 
\begin{enumerate}
\item the sectional curvature of $g$ is bounded in absolute value by some $\kappa > 0$ and
\item there exists $A > 0$ such that 
\begin{align*}
\sqrt{g_z(v,v)} \leq A\frac{\norm{v}}{\delta_\Omega(z)}
\end{align*}
for all $z \in \Omega$ and $v \in \Cb^d$. 
\end{enumerate}
\end{definition}

We will also assume that the boundary satisfies a weak accessibility condition. Given $z_0 \in \Cb^d$, $v \in \Cb^d$ with  $\norm{v}=1$, $\theta \in (0,\pi/2]$, and $r > 0$ define the truncated cone:
\begin{align*}
\Cc(z_0,v,\theta,r) = \{ z \in \Cb^d: 0< \norm{z-z_0} < r, \ \angle(z-z_0, v) < \theta \}.
\end{align*}

\begin{definition} If $\Omega \subset \Cb^d$ is a domain and $\xi \in \partial \Omega$, then we say \emph{$\partial\Omega$ satisfies an interior cone condition at $\xi$ with parameters $\theta \in (0,\pi/2]$ and $r>0$} if there exists $v \in \Cb^d$ with $\norm{v}=1$ such that $\Cc(\xi,v,\theta,r) \subset \Omega$.
\end{definition}

Our second main result is the following. 

\begin{theorem}\label{thm:main_intro_biholo} (see Theorem~\ref{thm:main}) Suppose $\Omega \subset \Cb^d$ is a bounded domain, $\varphi \in \Aut(\Omega)$,  $\partial \Omega$ satisfies an interior cone condition at $\xi_0 \in \partial \Omega$ with parameter $\theta$, and there exists an $\varphi$-invariant K{\"a}hler metric $g$ on $\Omega$ with property-$(BG)$ with parameters $\kappa, A$.

If 
\begin{align*}
L > 4d+2+\frac{\sqrt{\kappa} A}{\sin(\theta)}
\end{align*}
and
\begin{align*}
 \varphi(z) = z + { \rm O} \left( \norm{z-\xi_0}^L \right),
\end{align*}
then $\varphi = \id$.
\end{theorem}

\begin{remark}\label{rmk:bound_on_L} \ \begin{enumerate}
\item We will prove a slightly more general result in Theorem~\ref{thm:main} below. 
\item Notice that the Theorem does not assume that $\partial \Omega$ has any regularity (beyond the interior cone condition at $\xi$) and we do not even assume that $\varphi$ extends continuously to the boundary.
\item In the case when the injectivity radius of $(\Omega, g)$ is positive we can choose
\begin{align*}
L >  2+\frac{\sqrt{\kappa} A}{\sin(\theta)}.
\end{align*}
\end{enumerate}
\end{remark}

Based on Theorem~\ref{thm:main_intro_biholo} it seems natural to ask: 

\begin{question} If $\Omega \subset \Cb^d$ is a bounded pseudoconvex domain with finite type, does there exists a $\Aut(\Omega)$-invariant complete K{\"a}hler metric on $\Omega$ with property-$(BG)$?
\end{question} 

We should note that McNeal~\cite{McNeal1989} showed that the Bergman metric has bounded sectional curvature on any bounded pseudoconvex domain with finite type in $\Cb^2$. 

\subsection{Examples:}  Every bounded pseudoconvex domain $\Omega \subset \Cb^d$ has a unique complete K{\"a}hler-Einstein metric $g_\Omega$ with Ricci curvature $-1$. This was constructed by Cheng and Yau~\cite{CY1980} when $\Omega$ has $C^2$ boundary and by Mok and Yau~\cite{MY1983} in general. In this subsection we describe two situations where this metric has property-$(BG)$. 

\subsubsection{HHR domains}

Following Liu, Sun, and Yau~\cite{LSY2004a,LSY2004b}, a domain $\Omega$ is said to be \emph{holomorphic homogeneous regular (HHR)}  if there exists $s>0$ with the following property: for every $z \in \Omega$ there exists a holomorphic embedding $\varphi: \Omega \rightarrow \Cb^d$ such that $\varphi(z)=0$ and
\begin{align*}
s \Bb_d \subset \varphi(\Omega) \subset \Bb_d
\end{align*}
where $\Bb_d \subset \Cb^d$ is the unit ball. In the literature, a HHR domain is sometimes called a domain with the \emph{uniform squeezing property}, see for instance~\cite{Y2009}.

Examples of HHR domains include:
\begin{enumerate}
\item $\Tc_{g,n}$, the Teichm{\"u}ller space of hyperbolic surfaces with genus $g$ and $n$ punctures~\cite{LSY2004a}, 
\item bounded convex domains or more generally bounded $\Cb$-convex domains \cite{F1991, KZ2016, NA2017}, 
\item bounded domains where $\Aut(\Omega)$ acts co-compactly on $\Omega$, and
\item strongly pseudoconvex domains~\cite{DFW2014,DGZ2016}.
\end{enumerate}

Every HRR domain is pseudoconvex~\cite[Theorem 1]{Y2009} but not every pseudoconvex domain is an HRR domain. For instance, Forn{\ae}ss and Rong have constructed smoothly bounded pseudoconvex domains in $\Cb^3$ which are not HRR~\cite{FR2018}.

Results of S.K. Yeung~\cite{Y2009} imply that the K{\"a}hler-Einstein metric on a HRR domain has property-$(BG)$, see Section~\ref{sec:examples} for details, and so we have the following corollary of Theorem~\ref{thm:main}. 

\begin{theorem}\label{thm:unif_sq} Suppose $\Omega \subset \Cb^d$ is a bounded HRR domain and $\xi_0 \in \partial \Omega$ satisfies an interior cone condition. Then there exists $L >0$ such that: if $\varphi \in \Aut(\Omega)$ and
\begin{align*}
 \varphi(z) = z + { \rm O} \left( \norm{z-\xi_0}^L \right),
\end{align*}
then $\varphi = \id$. Moreover, we can choose $L$ to depend only on: the dimension $d$, the $s$ in the definition of a HRR domain, and the $\theta$ in the definition of interior cone condition.  
\end{theorem}

Every bounded convex domain $\Omega \subset \Cb^d$ is a HRR domain and work of Frankel~\cite{F1991} (also see~\cite{KZ2016, NA2017}) implies that for every $d \in \Nb$ there exists $s_d >0$ such that:  if $\Omega \subset \Cb^d$ is a bounded convex domain, then the HRR parameter of $\Omega$ is bounded below by $s_d$. So we have the following partial answer to Conjecture~\ref{conj:convex}.

\begin{corollary}\label{cor:convex1} For every $d,\theta > 0$ there exists $L=L(d,\theta)>0$ such that: if $\Omega \subset \Cb^d$ is a bounded convex domain, $\varphi \in \Aut(\Omega)$, $\partial \Omega$ satisfies an interior cone condition at $\xi_0 \in \partial \Omega$ with parameter $\theta$, and
\begin{align*}
 \varphi(z) = z + { \rm O} \left( \norm{z-\xi_0}^{L} \right),
\end{align*}
then $\varphi = \id$.
\end{corollary}

\subsubsection{Pinched negative curvature} 

Let $(M,J)$ be a complex manifold with K{\"a}hler metric $g$ and let $R$ denote the curvature tensor of $(M,g)$. Then the \emph{holomorphic sectional curvature} of a non-zero $X \in T_pM$ is given by 
\begin{align*}
H(g)(X) = \frac{R(X,JX,X, JX)}{g(X,X)g(X,X)}.
\end{align*}

Using work of Wu and Yau~\cite{WY2017}, see Section~\ref{sec:examples}, we will establish the following variant of Theorem~\ref{thm:main}. 

\begin{theorem}\label{thm:WYcor} Suppose $\Omega \subset \Cb^d$ is a bounded domain and there exists a complete K{\"a}hler metric $g$ on $\Omega$ such that 
\begin{align*}
-a \leq H(g) \leq -b
\end{align*}
for some $a,b > 0$. Assume $\partial \Omega$ satisfies an interior cone condition at $\xi_0 \in \partial \Omega$. Then there exists $L > 0$ such that: if $\varphi \in \Aut(\Omega)$ and 
\begin{align*}
 \varphi(z) = z + { \rm O} \left( \norm{z-\xi_0}^L \right),
\end{align*}
then $\varphi = \id$. Moreover, we can choose $L$ to depend only on:  $d$, $a$, $b$, and the $\theta$ in the definition of interior cone condition. 
\end{theorem}

\begin{remark} In it worth noting that the metric $g$ in Theorem~\ref{thm:WYcor} is not assumed to be $\Aut(\Omega)$-invariant. \end{remark}

\subsection{Sketch of the proofs:} The proofs of Theorems~\ref{thm:main_convex} and~\ref{thm:main_intro_biholo} use very different techniques: the former relies on Lempert's theory of complex geodesics while the latter uses tools from Riemannian geometry. However, similar ideas are used in both. In this section, we sketch the proof of Theorem~\ref{thm:main_intro_biholo} and then describe some of the ideas used to prove Theorem~\ref{thm:main_convex}. 

\subsubsection{Sketch of the proof of Theorem~\ref{thm:main_intro_biholo}:} The central idea in the proof is that curvature controls how fast geodesics can spread apart. For simplicity we will only describe the argument in the special case where $g$ is a K{\"a}hler metric with positive injectivity radius and 
\begin{align}
\label{eq:curvature_bd_intro}
\sup\{ \abs{\nabla^q R} : x \in \Omega, q=0,1,2\} < \infty
\end{align}
where $R$ is the curvature tensor of $g$. 

Let $d_\Omega$ denote the distance induced by $g$. In this case, we prove that there exists $C_1, \tau > 0$ such that: if $\gamma_1, \gamma_2 : [0,\infty) \rightarrow \Omega$ are unit speed geodesics and $0 < \epsilon < \tau$, then 
\begin{align}
\label{eq:sketch_of_proof}
d_\Omega(\gamma_1(t), \gamma_2(t)) \leq \frac{C_1}{\epsilon} \exp\left( \frac{\kappa+1}{2}t \right)\max_{t \in [0,\epsilon]} d_{\Omega}(\gamma_1(t), \gamma_2(t))
\end{align}
for $t > 0$ (see Proposition~\ref{prop:geod_spread} and Theorem~\ref{thm:dist_est} below). 

Using the interior  cone condition and the upper bound on $g$, we find a sequence of points $p_n$ converging to $\xi_0$ such that 
\begin{align*}
d_\Omega(p_n, p_0) \leq \frac{A}{\sin(\theta)} \log \frac{1}{\norm{p_n-\xi_0}}.
\end{align*}
We then fix a point $z_0 \in \Omega$ and consider unit speed geodesics $\gamma_n : [0,T_n] \rightarrow \Omega$ with $\gamma_n(0) = p_n$ and $\gamma_n(T_n)=z_0$. Using the interior cone condition, the upper bound on $g$, and the fact that 
\begin{align*}
 \varphi(z) = z + { \rm O} \left( \norm{z-\xi_0}^L \right),
\end{align*}
we show that there exists $\epsilon_n, C_2 > 0$ such that
\begin{align*}
\max_{t \in [0,\epsilon_n]} d_{ \Omega}(\gamma_n(t), (\varphi \gamma_n)(t) ) \leq C_2 \norm{p_n-\xi_0}^{L-1}
\end{align*}
and $\epsilon_n \geq \norm{p_n-\xi_0}/C_2$. Then from Equation~\eqref{eq:sketch_of_proof} we have
\begin{align*}
d_\Omega(z_0, \varphi(z_0))=d_{\Omega}(\gamma_n(T_n), (\varphi \gamma_n)(T_n)) \leq C^2 \exp\left( \frac{\kappa+1}{2}T_n\right)\norm{p_n-\xi_0}^{L-2}.
\end{align*}
However, 
\begin{align*}
T_n \leq d_\Omega(p_0, z_0) + \frac{A}{\sin(\theta)} \log \frac{1}{\norm{p_n-\xi_0}}
\end{align*}
and $\norm{p_n-\xi_0} \rightarrow 0$. So if 
\begin{align*}
L > 2+\frac{(\kappa+1)A}{2\sin(\theta)},
\end{align*}
then $d_\Omega(z_0, \varphi(z_0))=0$. Hence $\varphi(z_0) = z_0$. Since $z_0 \in \Omega$ was arbitrary, this implies that $\varphi = \id$. 

This argument actually shows that any
\begin{align*}
L > 2+\frac{\sqrt{\kappa} A}{\sin(\theta)}
\end{align*}
suffices. One simply replaces $g$ with $\kappa g$. Then repeating the above argument shows that if
\begin{align*}
L > 2+\frac{\left(\frac{\kappa}{\kappa}+1\right) \sqrt{\kappa} A}{2\sin(\theta)} = 2+\frac{\sqrt{\kappa} A}{\sin(\theta)}
\end{align*}
then $\varphi = \id$. 

When the injectivity radius of $(\Omega, g)$ is not assumed to be positive, some of the estimates are worse which forces us to assume that 
\begin{align*}
L > 4d+2+\frac{\sqrt{\kappa} A}{\sin(\theta)}.
\end{align*}
When $g$ does not satisfy Equation~\eqref{eq:curvature_bd_intro}, we use classical results about the Ricci flow to deform $g$ to obtain a metric that does, see Section~\ref{sec:deform} for details. 

The most difficult part of the argument is establishing the estimate in Equation~\eqref{eq:sketch_of_proof}. This requires a number of results about Riemannian manifolds which are discussed in Sections~\ref{sec:unit_tangent_bundle}, \ref{sec:two_lower_bounds}, \ref{sec:deform}, and \ref{sec:dist_est}.  

\subsubsection{Ideas in the proof of Theorem~\ref{thm:main_convex}} Like Burns and Krantz's proof of Theorem~\ref{thm:BK}, we study complex geodesics and their images under $f$. For strongly convex domains, complex geodesics are very well understood thanks to Lempert's deep work~\cite{L1981, L1982, L1984}. However, for convex domains with $C^2$ boundary and no finite type assumptions, complex geodesics are less understood and can have unpleasant behavior. For example, it is possible for a complex geodesic to not extend continuously to the boundary (see~\cite[Example 1.2]{B2016}). A key part of the proof of Theorem~\ref{thm:main_convex} is establishing some new results about complex geodesics which gives us some control over their behavior. The results are somewhat technical and we delay further discussion until Section~\ref{sec:GP}.

%
%

A second key part in the proof is a recent estimate of Christodoulou and Short~\cite{CS2018}. Before stating their result we need some notation: let $K_{\Db}$ is the Kobayashi distance on $\Db$ and let $B_{\Db}(z;r)$ be the open metric ball centered at $z \in \Db$ of radius $r > 0$ in $(\Db, K_{\Db})$.

\begin{theorem}[Christodoulou-Short~\cite{CS2018}]\label{thm:CS_intro} There exists $C > 0$ such that: if $f : \Db \rightarrow \Db$ is holomorphic, $z_0 \in \Db$, and $0 < \epsilon < 1$, then
\begin{align*}
K_{\Db}(f(z), z) \leq \frac{C}{\epsilon} \exp\Big(4 K_{\Db}(z_0,z) \Big)\sup_{w \in B_{\Db}(z_0;\epsilon)} K_{\Db}(f(w), w)
\end{align*}
for all $z \in \Db$. 
\end{theorem}

\begin{remark} Christodoulou and Short actually proved a stronger estimate, see Theorem~\ref{thm:CS} for the precise statement. \end{remark}

As we explain in Section~\ref{sec:unit_disk}, this estimate can be used to give a new proof of the Burns-Krantz theorem for the unit disk. We will use Theorem~\ref{thm:CS_intro} to prove Theorem~\ref{thm:main_convex} in a similar way to how Equation~\eqref{eq:sketch_of_proof} is used to prove Theorem~\ref{thm:main_intro_biholo}. 

\subsection{Notations}

\begin{enumerate}
\item For $z \in \Cb^d$ let $\norm{z}$ be the standard Euclidean norm and $d_{\Euc}(z_1, z_2) = \norm{z_1-z_2}$ be the standard Euclidean distance. 
\item For $z_0 \in \Cb^d$ and $r > 0$ let 
\begin{align*}
\Bb_d(z_0;r) = \left\{ z \in \Cb^d : \norm{z-z_0} < r\right\}.
\end{align*}
Then let $\Bb_d  = \Bb_d(0;1)$ and $\Db = \Bb_1$. 
\item Given an open set $\Omega \subset \Cb^d$, $p \in \Omega$, and $v \in \Cb^d \setminus \{0\}$ let 
\begin{align*}
\delta_{\Omega}(p)= \inf \left\{ d_{\Euc}(p,x) : x\in \partial \Omega \right\}.
\end{align*}
\end{enumerate}

 \subsection*{Acknowledgements} This material is based upon work supported by the National Science Foundation under grant DMS-1760233.

\part{The Proof of Theorem~\ref{thm:main_convex}}

 \section{Preliminaries for the proof of Theorem~\ref{thm:main_convex}}
 
 \subsection{The Kobayashi metric}

 In this expository section we recall the definition of the Kobayashi metric. A nice introduction to the Kobayashi metric and its properties can be found in~\cite{K2005} or~\cite{A1989}. 

Given a domain $\Omega \subset \mathbb{C}^d$ the \emph{(infinitesimal) Kobayashi metric} is the pseudo-Finsler metric
\begin{align*}
k_{\Omega}(x;v) = \inf \left\{ \abs{\xi} : f \in \Hol(\mathbb{D}, \Omega), \ f(0) = x, \ d(f)_0(\xi) = v \right\}.
\end{align*}
By a result of Royden~\cite[Proposition 3]{R1971} the Kobayashi metric is an upper semicontinuous function on $\Omega \times \mathbb{C}^d$. In particular, if $\sigma:[a,b] \rightarrow \Omega$ is an absolutely continuous curve (as a map $[a,b] \rightarrow \mathbb{C}^d$), then the function 
\begin{align*}
t \in [a,b] \rightarrow k_\Omega(\sigma(t); \sigma^\prime(t))
\end{align*}
is integrable and we can define the \emph{length of $\sigma$} to  be
\begin{align*}
\ell_\Omega(\sigma)= \int_a^b k_\Omega(\sigma(t); \sigma^\prime(t)) dt.
\end{align*}
One can then define the \emph{Kobayashi pseudo-distance} to be
\begin{equation*}
 K_\Omega(x,y) = \inf \left\{\ell_\Omega(\sigma) : \sigma\colon[a,b]
 \rightarrow \Omega \text{ is abs. cont., } \sigma(a)=x, \text{ and } \sigma(b)=y\right\}.
\end{equation*}
This definition is equivalent to the standard definition using analytic chains by a result of Venturini~\cite[Theorem 3.1]{Ven1989}.

When $\Omega$ is a bounded domain, $K_\Omega$ is a non-degenerate distance. For general domains there is no known characterization of when the Kobayashi distance is proper, but for convex domains we have the following result of Barth.

\begin{theorem}\cite{B1980}\label{thm:barth}
Suppose $\Omega$ is a convex domain. Then the following are equivalent:
\begin{enumerate}
\item $\Omega$ does not contain any complex affine lines,
\item $K_\Omega$ is a non-degenerate distance on $\Omega$, 
\item $(\Omega, K_\Omega)$ is a proper metric space, 
\item $(\Omega, K_\Omega)$ is a proper geodesic metric space. 
\end{enumerate}
\end{theorem}

One of the most important properties of the Kobayashi metric is the following distance decreasing property (which is immediate from the definition).

\begin{proposition} Suppose $\Omega_1 \subset \Cb^{d_1}$ and $\Omega_2 \subset \Cb^{d_2}$ are domains. If $f : \Omega_1 \rightarrow \Omega_2$ is a holomorphic map, then 
\begin{align*}
k_{\Omega_2}(f(z); d(f)_z(v)) \leq k_{\Omega_1}(z;v)
\end{align*}
for all $z \in \Omega_1$ and $v \in \Cb^d$. In particular, 
\begin{align*}
K_{\Omega_2}(f(z), f(w)) \leq K_{\Omega_1}(z,w)
\end{align*}
for all $z,w \in\Omega_1$. 
\end{proposition}

We will also frequently use the following elementary estimate (which follows from considering holomorphic maps of the form $\Bb_d \hookrightarrow \Omega$). 

\begin{observation} If $\Omega \subset \Cb^d$ is a domain, then 
\begin{align*}
k_{\Omega}(z;v) \leq \frac{ \norm{v}}{\delta_\Omega(z)}
\end{align*}
for all $z \in \Omega$ and $v \in \Cb^d$. 
\end{observation}

Finally, we make the following definition.

\begin{definition}
 For a domain $\Omega \subset \Cb^d$, $z_0 \in \Omega$, and $R \geq 0$ let $B_\Omega(z_0;R)$ be the open metric ball of radius $R$ centered at $z_0$ with respect to the Kobayashi metric, that is 
 \begin{align*}
 B_\Omega(z_0;R) = \{ z \in \Omega : K_\Omega(z,o) < R\}.
 \end{align*}
 \end{definition}

 \subsection{A quantitative estimate for maps of the disk}\label{subsec:CS_result}
 
 If $f : \Db \rightarrow \Db$ is holomorphic and fixes two distinct points in $\Db$, then the Schwarz lemma implies that $f=\id$. Recently, Christodoulou and Short established the following quantitative version of this uniqueness result. 

\begin{theorem}[Christodoulou-Short~\cite{CS2018}]\label{thm:CS} Suppose $f : \Db \rightarrow \Db$ is a holomorphic map and $a,b,z \in \Db$ with $a \neq b$. Then 
\begin{align*}
K_{\Db}(f(z), z) \leq C \Big( K_{\Db}(f(a),a) + K_{\Db}(f(b),b) \Big)
\end{align*}
where
\begin{align*}
C = \frac{1}{2K_{\Db}(a,b)} \exp \Big( 2K_{\Db}(z,a)+2K_{\Db}(z,b) + 2K_{\Db}(a,b) \Big).
\end{align*}
\end{theorem}

As a corollary we have the following. 

\begin{proposition}\label{prop:quant_id} Suppose that $f_n : \Db \rightarrow \Db$ are holomorphic maps, $z_n \in \Db$, and $0<r_n <1$. If 
\begin{align*}
\lim_{n \rightarrow \infty} \frac{e^{4K_{\Db}(z_n, 0)}}{r_n} \sup_{w \in B_{\Db}(z_n;r_n)} K_{\Db}(f_n(w), w)= 0,
\end{align*}
then $f_n$ converges locally uniformly to the identity map. 
\end{proposition}

\begin{proof}
Pick $a_n, b_n \in \Db$ such that 
\begin{align*}
K_{\Db}(a_n,z_n) = K_{\Db}(b_n, z_n) = \frac{1}{2} K_{\Db}(a_n, b_n) = r_n.
\end{align*}
Then fix some point $z \in \Db$. Then 
\begin{align*}
K_{\Db}(f_n(z), z) \leq C_n \Big( K_{\Db}(f_n(a_n),a_n) + K_{\Db}(f_n(b_n),b_n) \Big)
\end{align*}
where
\begin{align*}
C_n = \frac{1}{2K_{\Db}(a_n,b_n)} \exp \Big( 2K_{\Db}(z,a_n)+2K_{\Db}(z,b_n) + 2K_{\Db}(a_n,b_n) \Big).
\end{align*}
By construction
\begin{align*}
K_{\Db}(f(a_n),a_n) + K_{\Db}(f(b_n),b_n) \leq 2 \sup_{w \in B_{\Db}(z_n;r_n)} K_{\Db}(f_n(w), w)
\end{align*}
and
\begin{align*}
C_n \leq \frac{1}{4r_n} \exp \Big( 4K_{\Db}(z,z_n)+  8r_n \Big) \leq \frac{C}{4r_n} \exp \Big( 4K_{\Db}(0,z_n) \Big)
\end{align*}
where $C = \exp \Big( 4K_{\Db}(z,0)+8\Big)$.

Hence
\begin{align*}
K_{\Db}(f_n(z), z) \leq \frac{C}{2} \frac{e^{4K_{\Db}(z_n, 0)}}{r_n}   \sup_{w \in B_{\Db}(z_n;r_n)} K_{\Db}(f_n(w), w).
\end{align*}
So 
\begin{align*}
\lim_{n \rightarrow \infty} K_{\Db}(f_n(z), z) = 0.
\end{align*}
Since $z \in \Db$ was arbitrary we see that $f_n \rightarrow \id$. 

\end{proof}
 
\subsection{Complex geodesics in convex domains}

A holomorphic map $\varphi : \Db \rightarrow \Omega$ is called a \emph{complex geodesic} if 
\begin{align*}
K_{\Omega}(\varphi(z), \varphi(w)) = K_{\Db}(z,w)
\end{align*}
for all $z,w \in \Db$. A \emph{left inverse of a complex geodesic $\varphi$} is a holomorphic map $\pi : \Omega \rightarrow \Db$ such that $\pi \circ \varphi = \id$. In this section we recall a result of Lempert.

\begin{theorem}[Lempert~\cite{L1981, L1982,L1984}]\label{thm:lempert} Suppose that $\Omega$ is a strongly convex domain with $C^\infty$ boundary. If $z, w\in \Omega$ are distinct, then there exists a unique complex geodesic $\varphi : \Db \rightarrow \Omega$ with $z,w \in \varphi(\Db)$. Further, $\varphi$ has a left inverse $\pi$ and for every $\zeta \in \Db$ \begin{align*}
\pi^{-1}(\zeta) = \Omega \cap H_{\zeta}
\end{align*}
for some complex hyperplane plane $H_{\zeta} \subset \Cb^d$.
\end{theorem}

The fact that 
\begin{align*}
\pi^{-1}(\zeta) = \Omega \cap H_{\zeta}.
\end{align*}
for some complex hyperplane plane $H_{\zeta}$ follows from the description of $\pi$ given in the proof of the Lemma in~\cite{L1982}. In particular, if $\wt{\varphi}: \Db \rightarrow \Omega$ is the \emph{dual map of $\varphi$}, then $\pi(z) \in \Db$ is the unique solution to the equation 
\begin{align*}
\left[ z-\varphi(\zeta), \wt{\varphi}(\zeta) \right]=0
\end{align*}
where $[a,b] = \sum_{i=1}^d a_i b_i$. Thus 
\begin{align*}
\pi^{-1}(\zeta) = \Omega \cap \left\{ z \in \Cb^d: \left[ z-\varphi(\zeta), \wt{\varphi}(\zeta) \right]=0\right\}.
\end{align*}

If $\Omega$ is a bounded convex domain, then $\Omega$ can be written as an increasing union of smoothly bounded strongly convex domain. Thus Montel's theorem implies the following corollary of Lempert's theorem. 

\begin{corollary}Suppose that $\Omega$ is a bounded convex domain. If $z, w\in \Omega$ are distinct, then there exists a complex geodesic $\varphi : \Db \rightarrow \Omega$ with $z,w \in \varphi(\Db)$. Further, $\varphi$ has a left inverse $\pi$ such that for every $\zeta \in \Db$ 
\begin{align*}
\pi^{-1}(\zeta) = \Omega \cap H_{\zeta}
\end{align*}
for some complex hyperplane plane $H_{\zeta}$.
\end{corollary}

\begin{remark} For a general convex domain $\Omega$, it is possible for two points $z,w \in \Omega$ to be contained in many different complex geodesics. 
\end{remark}

The left inverses with this hyperplane preimage property play a fundamental role in the proof of Theorem~\ref{thm:main_convex} and so we make the following definition. 

\begin{definition} Suppose that $\Omega$ is a bounded convex domain and $\varphi : \Db \rightarrow \Omega$ is a complex geodesic. Then we say $\pi: \Omega \rightarrow \Db$ is a \emph{good left inverse of $\varphi$} if $\pi$ is a left inverse of $\varphi$ and for every $\zeta \in \Db$ 
\begin{align*}
\pi^{-1}(\zeta) = \Omega \cap H_{\zeta}.
\end{align*}
for some complex hyperplane plane $H_{\zeta}$.
\end{definition}

\section{The Gromov product and complex geodesics}\label{sec:GP}

In a metric space $(X,d)$, the Gromov product of $x,y \in X$ at $z \in X$ is defined to be
 \begin{align*}
 (x|y)_z = \frac{1}{2} \left( d(x,z) +d(z,y) - d(x,y) \right).
 \end{align*}
 When $(X,d)$ is a proper geodesic Gromov hyperbolic metric space, there is a compactification $X \cup X(\infty)$ of $X$, called the \emph{ideal boundary}, with the following property.
 
 \begin{proposition} Suppose $(X,d)$ is a proper geodesic Gromov hyperbolic metric space. Suppose $x_m, y_n$ are sequences in $X$ such that $x_m \rightarrow \xi \in X(\infty)$ and $y_n \rightarrow \eta \in X(\infty)$. Then $\xi = \eta$ if and only if 
 \begin{align*}
 \lim_{m,n \rightarrow \infty} (x_m |y_n)_{z} = \infty
 \end{align*}
 for any $z \in X$. 
 \end{proposition}
 
 For the Kobayashi metric on convex domains the Gromov product behaves almost as nicely near the topological boundary. Given a domain $\Omega \subset \Cb^d$ we define the \emph{Gromov product} of points $z,w,o \in \Omega$ to be 
\begin{align*}
(z|w)_o^{\Omega} = \frac{1}{2} \left( K_\Omega(z,o) +K_\Omega(o,w)-K_\Omega(z,w) \right).
\end{align*}

We also need the following definition.

\begin{definition} Given a convex domain $\Omega \subset \Cb^d$ with $C^1$ boundary and $x \in \partial \Omega$ let $H_x \partial \Omega$ denote the unique complex affine hyperplane tangent to $\partial \Omega$ at $x$. 
\end{definition}

\begin{remark} Since $\Omega$ is convex, if $x \in \partial \Omega$, then $\partial \Omega \cap H_x\partial \Omega$ is a closed convex set which is sometimes called the \emph{closed complex face of $\partial \Omega$ containing $x$}.
\end{remark}

We then have the following. 

\begin{theorem}\label{thm:gromov_prod}\cite[Theorem 4.1]{Z2017} 
Suppose $\Omega \subset \Cb^d$ is a bounded convex domain with $C^{1,\alpha}$ boundary and $p_n, q_m \in \Omega$ are sequences such that $p_n \rightarrow x \in \partial \Omega$ and $q_m \rightarrow y \in \partial \Omega$. 
\begin{enumerate}
\item If $x=y$, then 
\begin{align*}
\lim_{n,m \rightarrow \infty} ( p_n | q_m)_o^{\Omega} = \infty.
\end{align*}
\item If 
\begin{align*}
\limsup_{n,m \rightarrow \infty} \ ( p_n | q_m)_o^{\Omega} = \infty,
\end{align*}
then $H_{x} \partial \Omega = H_{y} \partial \Omega$.
\end{enumerate}
\end{theorem}

In~\cite{Z2017}, the behavior of the Gromov product was used to understand holomorphic self maps of $\Omega$ and real geodesics in $(\Omega, K_\Omega)$. In this section, we adapt those arguments to study the behavior of complex geodesics. 

Our first application establishes a boundary extension property of complex geodesics. For a smooth strongly convex domain $\Omega \subset \Cb^d$, Lempert~\cite{L1981} showed that every complex geodesic $\varphi : \Db \rightarrow \Omega$ extends to a smooth map $\overline{\Db} \rightarrow \overline{\Omega}$. However, this fails when $\Omega$ is not strongly convex: there exist examples of smoothly bounded convex domains $\Omega \subset \Cb^d$ and complex geodesics $\varphi : \Db \rightarrow \Omega$ which do not even extend to a continuous map $\overline{\Db} \rightarrow \overline{\Omega}$, see for instance~\cite[Example 1.2]{B2016}. 

For a convex domain $\Omega$ with $C^1$ boundary define 
\begin{align*}
H(\partial \Omega) = \left\{ H_x \partial \Omega : x \in \partial \Omega\right\}.
\end{align*}
Then $H(\partial \Omega)$ is a closed subset of the Grassmanian of affine complex hyperplanes in $\Cb^d$. 

\begin{proposition}\label{prop:bd_maps}
Suppose $\Omega \subset \Cb^d$ is a bounded convex domain with $C^{1,\alpha}$ boundary. If $\varphi : \Db \rightarrow \Omega$ is a complex geodesic, then there exists a continuous map $\wh{\varphi} : \partial \Db \rightarrow H(\partial \Omega)$ such that 
\begin{align*}
\lim_{z \rightarrow \zeta} d_{\Euc}\left(\varphi(z), \partial \Omega \cap \wh{\varphi}(\zeta) \right) = 0
\end{align*}
for every $\zeta \in \partial \Db$.
\end{proposition}

\begin{proof}
Suppose $\zeta \in \partial \Db$, then 
\begin{align*}
\lim_{z,w \rightarrow \zeta} ( \varphi(z) | \varphi(w))_{\varphi(0)}^{\Omega} = \lim_{z,w \rightarrow \zeta} ( z | w)_{0}^{\Db} = \infty
\end{align*}
by applying Theorem~\ref{thm:gromov_prod} to $\Db$. So by Theorem~\ref{thm:gromov_prod} applied to $\Omega$, there exists some $\wh{\varphi}(\zeta) \in H(\partial \Omega)$ such that 
\begin{align*}
\lim_{z \rightarrow \zeta} d_{\Euc}\left(\varphi(z), \partial \Omega \cap \wh{\varphi}(\zeta) \right) = 0.
\end{align*}

It remains to show that the map $\wh{\varphi} : \partial \Db \rightarrow H(\partial \Omega)$ is continuous. Suppose that $\zeta_n \in \partial \Db$ converges to $\zeta \in \partial \Db$. We claim that 
\begin{align*}
\wh{\varphi}(\zeta_n)  \rightarrow \wh{\varphi}(\zeta).
\end{align*}
Since $\partial \Omega$ is compact, $H(\partial \Omega)$ is also compact. So it is enough to show that every convergent subsequence of $\wh{\varphi}(\zeta_n)$ converges to $\wh{\varphi}(\zeta)$. So without loss of generality we can assume that $\wh{\varphi}(\zeta_n)$ converges to some hyperplane $H \in H(\partial \Omega)$. 

For each $n$ pick a sequence $z_{n,m} \in \Db$ such that 
\begin{align*}
\lim_{m \rightarrow \infty} z_{n,m} = \zeta_n.
\end{align*}
Then we can pick $m_n$ such that $\abs{z_{n,m_n}-\zeta_n} < 1/n$ and 
\begin{align*}
d_{\Euc}(\varphi(z_{n,m_n}), \partial \Omega \cap \wh{\varphi}(\zeta_n) ) <1/n.
\end{align*}
Then $z_{n,m_n} \rightarrow \zeta$ and so 
\begin{align*}
\lim_{n \rightarrow \infty} d_{\Euc}(\varphi(z_{n,m_n}), \partial \Omega\cap\wh{\varphi}(\zeta)) =0. 
\end{align*}
But by our construction of $z_{n,m_n}$ we also have 
\begin{align*}
\lim_{n \rightarrow \infty} d_{\Euc}(\varphi(z_{n,m_n}), \partial \Omega \cap H) =0. 
\end{align*}
So we must have $\wh{\varphi}(\zeta)  = H$. 
\end{proof}

\begin{definition} Suppose $\Omega \subset \Cb^d$ is a bounded convex domain with $C^{1,\alpha}$ boundary and $\varphi : \Db \rightarrow \Omega$ is a complex geodesic. We call the map $\wh{\varphi} : \Db \rightarrow H(\Omega)$ in Proposition~\ref{prop:bd_maps} the \emph{hyperplane boundary extension of $\varphi$}. 
\end{definition} 

The next result shows that $\wh{\varphi}$ depends continuously on $\varphi$. 

\begin{proposition}\label{prop:convergence_at_infinity}
Suppose $\Omega \subset \Cb^d$ is a bounded convex domain with $C^{1,\alpha}$ boundary and $\varphi_n : \Db \rightarrow \Omega$ is a sequence of complex geodesics converging locally uniformly to a complex geodesic $\varphi : \Db \rightarrow \Omega$. 

If $z_n \in \Db$ converges to $\zeta \in \partial \Db$ and $\varphi_n(z_n) \rightarrow x \in \partial \Omega$, then 
\begin{align*}
\wh{\varphi}(\zeta) = H_x \partial \Omega.
\end{align*}
Further, if $\zeta_n \in \partial \Db$ converges to $\zeta \in \partial \Db$, then
\begin{align*}
\wh{\varphi}(\zeta)= \lim_{n \rightarrow \infty} \wh{\varphi}_n(\zeta_n).
\end{align*}
\end{proposition}

\begin{proof} Suppose $z_n \in \Db$ converges to $\zeta \in \partial \Db$ and $\varphi_n(z_n) \rightarrow x \in \partial \Omega$. Next fix $0 < r < 1$. Since $\varphi_n$ and $\varphi$ are complex geodesics 
\begin{align*}
K_\Omega( \varphi_n(z_n), \varphi_n(0)) = K_\Omega( \varphi_n(z_n), \varphi_n(rz_n))+K_\Omega( \varphi_n(rz_n), \varphi_n(0))
\end{align*}
and
\begin{align*}
K_\Omega( \varphi(z_n), \varphi(0)) = K_\Omega( \varphi(z_n), \varphi(rz_n))+K_\Omega( \varphi(rz_n), \varphi(0)).
\end{align*}
Thus, by the triangle inequality,
\begin{align*}
(\varphi_n(z_n) |  \varphi(z_n) )_{\varphi(0)}^\Omega \geq (\varphi_n(rz_n) |  \varphi(rz_n) )_{\varphi(0)}^\Omega - \frac{1}{2}K_\Omega(\varphi_n(0),\varphi(0)).
\end{align*}
Further, 
\begin{align*}
\lim_{n \rightarrow \infty} (\varphi_n(rz_n) |  \varphi(rz_n) )_{\varphi(0)}^\Omega = K_\Omega(\varphi(r\zeta),\varphi(0)) = K_{\Db}(r\zeta,0).
\end{align*}
So
\begin{align*}
\lim_{n \rightarrow \infty} (\varphi_n(z_n) |  \varphi(z_n) )_{\varphi(0)}^\Omega \geq K_{\Db}(r\zeta,0).
\end{align*}
Since $0 < r < 1$ was arbitrary we see that 
\begin{align*}
\lim_{n \rightarrow \infty} (\varphi_n(z_n) |  \varphi(z_n) )_{\varphi(0)}^\Omega = \infty.
\end{align*}
Then since $\varphi_n(z_n) \rightarrow x$ and
\begin{align*}
\lim_{n \rightarrow \infty} d_{\Euc}(\varphi(z_n), \partial \Omega \cap \wh{\varphi}(\zeta)) =0,
\end{align*}
Theorem~\ref{thm:gromov_prod} implies that we must have $H_{x} \partial \Omega = \wh{\varphi}(\zeta)$. 

Now we prove the ``further'' part of the Proposition. Suppose that $\zeta_n \in \partial \Db$ converges to $\zeta \in \partial \Db$. We claim that 
\begin{align*}
\wh{\varphi}(\zeta)= \lim_{n \rightarrow \infty} \wh{\varphi}_n(\zeta_n).
\end{align*}
Since $\partial \Omega$ is compact, $H(\partial \Omega)$ is also compact. So it is enough to show that every convergent subsequence of $\wh{\varphi}(\zeta_n)$ converges to $\wh{\varphi}(\zeta)$. So without loss of generality we can assume that $\wh{\varphi}(\zeta_n)$ converges to some hyperplane $H \in H(\partial \Omega)$. 

Now fix a sequence $r_n \nearrow 1$ such that 
\begin{align*}
d_{\Euc}(\varphi_n(r_n\zeta_n), \partial \Omega \cap \wh{\varphi}_n(\zeta_n)) <1/n.
\end{align*} 
By passing to a subsequence we can suppose that $\varphi_n(r_n\zeta_n) \rightarrow x \in \partial \Omega$. Then $H_x \partial \Omega = H$. So by the first assertion in the Proposition $H_{x} \partial \Omega = \wh{\varphi}(\zeta)$.

\end{proof}

Consider the one-parameter subgroup $\{ a_t : t \in \Rb\} \leq \Aut(\Db)$ given by
\begin{align*}
a_t(z) = \frac{ \cosh(t) z+\sinh(t)}{\sinh(t) z + \cosh(t)}.
\end{align*}
Then $t \rightarrow a_t(0)$ is a geodesic in $(\Db, K_{\Db})$ and so if $\varphi : \Db \rightarrow \Omega$ is a complex geodesic, then $t \rightarrow \varphi(a_t(0))$ is a geodesic in $(\Omega, K_\Omega)$. 

The next result shows that geodesic segments whose endpoints are near boundary points $x,y \in \partial \Omega$ with $H_x \partial \Omega \neq H_y \partial \Omega$ ``bend'' into $\Omega$. 

\begin{proposition}\label{prop:visible}
Suppose $\Omega \subset \Cb^d$ is a bounded convex domain with $C^{1,\alpha}$ boundary and $p_n, q_n \in \Omega$ are sequences such that $p_n \rightarrow x \in \partial \Omega$ and $q_n \rightarrow y \in \partial \Omega$ with $H_{x} \partial \Omega \neq H_{y} \partial \Omega$. 

If $\varphi_n:\Db \rightarrow \Omega$ is a complex geodesic with $\varphi_n(0)=q_n$ and $\varphi_n(t_n) = p_n$ where $0 < t_n < 1$, then there exists $n_k \rightarrow \infty$ and $s_k \in [0,t_{n_k}]$ so that the complex geodesics $\varphi_{n_k} \circ a_{s_k}$ converge locally uniformly to a complex geodesic $\varphi:\Db \rightarrow\Omega$. Moreover, 
\begin{align*}
\lim_{z \rightarrow -1} d_{\Euc}(\varphi(z), \partial \Omega \cap H_{x} \partial \Omega) = 0
\end{align*}
and 
\begin{align*}
\lim_{z \rightarrow 1} d_{\Euc}(\varphi(z), \partial \Omega \cap H_{y} \partial \Omega) = 0.
\end{align*}
\end{proposition}

\begin{proof} Since $H_{x} \partial \Omega \neq H_{y} \partial \Omega$ there exists open neighborhoods $U_x$ of $\partial \Omega \cap H_x \partial\Omega$ and $U_y$ of $\partial \Omega \cap H_y \partial\Omega$ such that $\overline{U_x} \cap \overline{U_y} = \emptyset$. 

For $n$ large, $\varphi_n(0) \in U_x$ and $\varphi_n(T_n) \in U_y$. So there exists some $s_n \in (0,T_n)$ such that 
\begin{align*}
\varphi_n(s_n) \in \Omega \setminus (U_x \cup U_y).
\end{align*}
By passing to a subsequence we can suppose that $\varphi_n(s_n) \rightarrow z \in \overline{\Omega}$. \\

\noindent\textbf{Claim:} $z \in \Omega$. 

\begin{proof}[Proof of Claim:]Suppose not. Then $z \in \partial \Omega$. Since $z \notin U_x \cup U_y$, we see that $H_z \partial \Omega$ does not equal $H_x \partial \Omega$ or $H_y \partial \Omega$. Fix some $z_0 \in \Omega$. Let 
\begin{align*}
R_1 = \sup\{ (\varphi_n(s_n) | \varphi_n(0))_{z_0}^\Omega  : n=1,2,\dots \}
\end{align*}
and
\begin{align*}
R_2 = \sup\{ (\varphi_n(s_n) | \varphi_n(t_n))_{z_0}^\Omega  : n=1,2,\dots \}.
\end{align*}
By Theorem~\ref{thm:gromov_prod}, both $R_1$ and $R_2$ are finite. Then
\begin{align*}
K_\Omega(\varphi_n(0), & \varphi_n(t_n)) = K_\Omega(\varphi_n(0), \varphi_n(s_n))+K_\Omega(\varphi_n(s_n), \varphi_n(t_n)) \\
& \geq K_\Omega(\varphi_n(0),z_0)+2K_\Omega(z_0, \varphi_n(s_n))+K_\Omega(z_0, \varphi_n(t_n))-2R_1-2R_2.
\end{align*}
By the triangle inequality
\begin{align*}
K_\Omega(\varphi_n(0),  \varphi_n(t_n)) \leq K_\Omega(\varphi_n(0),z_0)+K_\Omega(z_0, \varphi_n(t_n))
\end{align*}
and so we see that 
\begin{align*}
K_\Omega(z_0, \varphi_n(s_n)) \leq R_1 + R_2.
\end{align*}
But this is impossible since $\varphi_n(s_n) \rightarrow z \in \partial \Omega$ and $K_\Omega$ is a proper distance. So we must have $z \in \Omega$.\end{proof}

Now since $z \in \Omega$, after possibly passing to a subsequence we can suppose that $\phi_n = \varphi_n \circ a_{s_n}$ converges to a complex geodesic $\varphi : \Db \rightarrow \Omega$. Now since $p_n=\phi_n(a_{-s_n}(0)) \rightarrow x$ and $q_n=\phi_n(a_{-s_n}(t_n)) \rightarrow y$, the previous Proposition implies that  
\begin{align*}
\lim_{z \rightarrow -1} d_{\Euc}(\varphi(z), \partial \Omega \cap H_{x} \partial \Omega) = 0
\end{align*}
and 
\begin{align*}
\lim_{z \rightarrow 1} d_{\Euc}(\varphi(z), \partial \Omega \cap H_{y} \partial \Omega) = 0.
\end{align*}
\end{proof}

\begin{proposition}\label{prop:inv_hyperplane_convergence} Suppose $\Omega \subset \Cb^d$ is a bounded convex domain with $C^{1,\alpha}$ boundary, $\varphi : \Db \rightarrow \Omega$ is a complex geodesic, and $\pi : \Omega \rightarrow \Db$ is a good left inverse of $\varphi$. For each $z \in \Db$, let $H_z$ denote the complex hyperplane such that 
\begin{align*}
\pi^{-1}(z) = \Omega \cap H_z.
\end{align*}
Then 
\begin{align*}
\wh{\varphi}(\zeta)=\lim_{z \rightarrow \zeta} H_z
\end{align*}
for every $\zeta \in \partial \Db$. 
\end{proposition}

\begin{proof} Fix some $\zeta \in \partial \Db$ and suppose for a contradiction that 
\begin{align*}
\lim_{z \rightarrow \zeta} H_z \neq \wh{\varphi}(\zeta).
\end{align*}
Then by compactness, we can find a sequence $z_n \in \Db$ converging to $\zeta$ such that 
\begin{align*}
\lim_{n \rightarrow \infty} H_{z_n} = H
\end{align*}
and $H \neq \wh{\varphi}(\zeta)$. By passing to another subsequence we can suppose that $\varphi(z_n) \rightarrow x \in \partial \Omega$. Then $x \in \wh{\varphi}(\zeta)$ and so $H_x \partial \Omega= \wh{\varphi}(\zeta)$. Now each $H_{z_n}$ is a complex hyperplane containing $\varphi(z_n)$. So $H$ is a complex hyperplane containing $x$. We next claim that $H \cap \Omega = \emptyset$. If not, then after passing to a subsequence there exists $w \in H \cap \Omega$ and $w_n \in H_{z_n} \cap \Omega$ such that $w_n \rightarrow w$. Then 
\begin{align*}
\zeta = \lim_{n \rightarrow \infty} z_n = \lim_{n \rightarrow \infty} \pi(w_n) = \pi(w)
\end{align*}
which is impossible because $\pi(\Omega) = \Db$. So $H \cap \Omega = \emptyset$. But then, since $\Omega$ is convex and $x \in H$, we have
\begin{align*}
H = H_x \partial \Omega =\wh{\varphi}(\zeta)
\end{align*}
which is a contradiction. 
\end{proof}

\section{The one dimensional case}\label{sec:unit_disk}

In this section we use Proposition~\ref{prop:quant_id} to provide a new proof of the Burns-Krantz theorem for the disc. The one dimensional result is not needed in the proof of Theorem~\ref{thm:main_convex}, but this simple case motivates the argument. 

\begin{theorem}[Burn-Krantz~\cite{BK1994}] Suppose $f : \Db \rightarrow \Db$ is holomorphic and there exists some $\xi_0 \in \partial \Db$ such that 
\begin{align*}
f(z) = z + { \rm o}\left( \abs{z-\xi_0}^3 \right),
\end{align*}
then $f= \id$. 
\end{theorem}

For the rest of the section suppose $f : \Db \rightarrow \Db$ is holomorphic and there exists some $\xi_0 \in \partial \Db$ such that 
\begin{align*}
f(z) = z + { \rm o}\left( \abs{z-\xi_0}^3 \right).
\end{align*}
Without loss of generality we can assume that $\xi_0=1$. Then there exists a non-decreasing function $E : [0,\infty) \rightarrow [0,\infty)$ such that 
\begin{align*}
\abs{f(z)-z} \leq E( \abs{z-1})
\end{align*}
and 
\begin{align*}
\lim_{ r \rightarrow 0} \frac{E(r)}{r^3} = 0.
\end{align*} 

Fix a sequence $0< r_n <1$ with $r_n \rightarrow 0$. Then consider the points $p_n = 1-r_n$. Then 
\begin{align*}
K_{\Db}(0,p_n) = \frac{1}{2} \log \frac{ 1+\abs{z}}{1-\abs{z}} \leq \frac{1}{2} \log \frac{2}{r_n}.
\end{align*}

From the well known explicit formula for the Kobayashi metric on $\Db$ we have 
\begin{align*}
\frac{\abs{v}}{2(1-\abs{z})} \leq k_{\Db}(z;v) \leq \frac{\abs{v}}{1-\abs{z}}
\end{align*}
for all $z \in \Db$ and $v \in \Cb$. Using this estimate the next two lemmas are simple exercises. 

\begin{lemma} There exists $C > 0$ such that: If $w \in \Bb(p_n; r_n/4)$, then 
\begin{align*}
K_{\Db}(w, f(w)) \leq \frac{C}{r_n} E(5r_n/4). 
\end{align*}
\end{lemma}

For each $n$ define 
\begin{align*}
\epsilon_n = \sup\{ \epsilon : B_{\Db}(p_n;\epsilon) \subset \Bb(p_n; r_n/4)\}.
\end{align*}

\begin{lemma} There exists some $a > 0$ such that $\epsilon_n \geq a$ for all $n$. \end{lemma}

Then
\begin{align*}
\lim_{n \rightarrow \infty} \frac{e^{4K_{\Db}(z_n, 0)}}{\epsilon_n} \sup_{w \in B_{\Db}(z_n;\epsilon_n)} K_{\Db}(f(w), w) \leq  \frac{4C}{a} \lim_{n \rightarrow \infty} \frac{1}{r_n^3} E(5r_n/4) = 0.
\end{align*}
So if we apply Proposition~\ref{prop:quant_id} to the constant sequence $f$, then we see that $f=\id$.

\section{Proof of Theorem~\ref{thm:main_convex}}\label{sec:pf_of_thm_main}

For the rest of the section suppose that $\Omega \subset \Cb^d$ is a bounded convex domain with $C^2$ boundary, $f : \Omega \rightarrow \Omega$ is holomorphic map, and there exists $\xi_0 \in \partial \Omega$ such that 
\begin{align*}
f(z) = z + { \rm o}\left( \norm{z-\xi_0}^4\right).
\end{align*}
Then there exists a non-decreasing function $E : [0,\infty) \rightarrow [0,\infty)$ such that 
\begin{align*}
\norm{f(z)-z} \leq E( \norm{z-\xi_0})
\end{align*}
and 
\begin{align*}
\lim_{ r \rightarrow 0} \frac{E(r)}{r^4} = 0.
\end{align*}

The key step in the proof is the following proposition. 

\begin{proposition}\label{prop:existence} For any $q \in \Omega$ there exists a complex geodesic $\varphi : \Db \rightarrow \Omega$ and a good left inverse $\pi : \Omega \rightarrow \Db$ such that $\varphi(0)=q$, $\pi \circ f \circ \varphi = \id$, and 
\begin{align*}
\lim_{z \rightarrow 1} d_{\Euc}\left(\varphi(z), \partial \Omega \cap H_{\xi_0}\partial \Omega\right) = 0.
\end{align*}
\end{proposition}

The proof of the proposition will require some lemmas. Let ${ \bf n}_\Omega(\xi_0)$ denote the inward pointing unit normal vector at $\xi_0$. Then consider a sequence 
\begin{align*}
p_n = \xi_0 + r_n { \bf n}_\Omega(\xi_0) \in \Omega
\end{align*}
which converges to $\xi_0$. Next  fix a point $z_0 \in \Omega$. Then by~\cite[Theorem 2.3.51]{A1989} there exists some $C_0 > 0$ such that 
\begin{align}
\label{eq:dist_est}
K_\Omega(z_0, p_n) \leq C_0 + \frac{1}{2} \log \frac{1}{r_n}.
\end{align}

\begin{lemma} There exists $C_1 > 0$ such that: If $w \in \Bb(p_n; r_n/4)$, then 
\begin{align*}
K_\Omega(w, f(w)) \leq \frac{C_1}{r_n} E(5r_n/4). 
\end{align*}
\end{lemma}

\begin{proof} Pick $N > 0$ such that 
\begin{align*}
E(5r_n/4) \leq r_n/4
\end{align*}
for all $n \geq N$.

If $n \geq N$ and $w \in \Bb(p_n; r_n/4)$, then 
\begin{align*}
\norm{f(w)-w} \leq E( \norm{w-\xi_0}) \leq E(5r_n/4) \leq r_n/4.
\end{align*}
Let $\sigma : [0,1] \rightarrow \Omega$ be the curve $\sigma(t) = (1-t)w+tf(w)$. Then
\begin{align*}
\delta_\Omega(\sigma(t)) \geq \delta_\Omega(p_n) - \norm{\sigma(t)-p_n} \geq r_n - r_n/2 = r_n/2
\end{align*}
for all $t \in [0,1]$. So
\begin{align*}
K_\Omega(w,f(w))
& \leq \ell_\Omega(\sigma) = \int_0^1 k_\Omega(\sigma(t); \sigma^\prime(t)) dt \leq \int_0^1 \frac{ \norm{\sigma^\prime(t)}}{\delta_\Omega(\sigma(t))}dt \\
& \leq \int_0^1 \frac{2 \norm{f(w)-w}}{r_n} dt \leq \frac{2}{r_n} E(5r_n/4).
\end{align*}

So there exists $C_1 > 0$ such that: If $w \in \Bb(p_n; r_n/4)$, then 
\begin{align*}
K_\Omega(w, f(w)) \leq \frac{C_1}{r_n} E(5r_n/4). 
\end{align*}
\end{proof}

For each $n$ define 
\begin{align*}
\epsilon_n = \sup\{ \epsilon : B_\Omega(p_n;\epsilon) \subset \Bb(p_n; r_n/4)\}.
\end{align*}

\begin{lemma}\label{lem:lower_bd_epsilon_n} There exists some $a > 0$ such that $\epsilon_n \geq a r_n$ for all $n$. \end{lemma}

\begin{proof} Since $\Omega$ is a bounded domain, there exists some $R > 0$ such that $\Omega \subset \Bb(0;R)$. Then 
\begin{align*}
k_\Omega(z;v) \geq k_{\Bb(0;R)}(z;v) \geq \frac{1}{R} \norm{v}
\end{align*}
for all $z \in \Omega$ and $v \in \Cb^d$. Then
\begin{align*}
K_\Omega(z,w) \geq \frac{1}{R} \norm{z-w}
\end{align*}
for all $z,w \in \Omega$. So if $a = 1/(4R)$ and $w \in B_\Omega(p_n;ar_n)$ then 
\begin{align*}
\norm{p_n-w} \leq R K_\Omega(p_n,w) \leq r_n/4.
\end{align*}
Hence $\epsilon_n \geq a r_n$ for all $n$. 
\end{proof}

\begin{proof}[Proof of Proposition~\ref{prop:existence}] By Theorem~\ref{thm:lempert}, for each $n$ there exists a complex geodesic $\varphi_n : \Db \rightarrow \Omega$ with a good left inverse $\pi_n: \Omega \rightarrow \Db$ such that $\varphi_n(0)=q$ and $\varphi_n(t_n) = p_n$ for some $t_n \in (0,1)$. By Montel's theorem and possibly passing to a subsequence we can assume that $\varphi_n$ and $\pi_n$ converge locally uniformly to holomorphic maps $\varphi$ and $\pi$. Then $\varphi$ is a complex geodesic, $\pi$ is a good left inverse of $\varphi$, $\varphi(0) = q$, and by Proposition~\ref{prop:convergence_at_infinity}
\begin{align*}
\lim_{z \rightarrow 1} d_{\Euc}\left(\varphi(z), \partial \Omega \cap H_{\xi_0} \partial \Omega\right) = 0.
\end{align*}

Define $F_n: \Db \rightarrow \Db$ by $F_n(z) = \pi_n \circ f \circ \varphi_n$. Then $F_n$ converges to $\pi \circ f \circ \varphi$. 

Suppose that $w \in B_{\Db}(t_n;\epsilon_n)$. Then $\varphi_n(w) \in B_{\Omega}(p_n;\epsilon_n)$ since $\varphi_n$ is a complex geodesic. Then 
\begin{align*}
K_{\Db}(F_n(w), w) 
&= K_{\Db}( \pi_n \circ f \circ \varphi_n(w), \pi_n \circ \varphi_n(w)) \\
& \leq K_\Omega(f(\varphi_n(w)), \varphi_n(w)) \leq \frac{C_1}{r_n} E(5r_n/4). 
\end{align*}
Further 
\begin{align*}
K_{\Db}(t_n, 0) = K_\Omega(p_n, q) \leq K_\Omega(p_n,z_0) + K_\Omega(z_0, q)
\end{align*}
so by Equation~\eqref{eq:dist_est}
\begin{align*}
e^{4K_{\Db}(t_n, 0)} \leq  A r_n^{-2}
\end{align*}
where $A = \exp(4 K_\Omega(z_0, q) + 4C_0)$. Thus 
\begin{align*}
\lim_{n \rightarrow \infty} \frac{e^{4K_{\Db}(t_n, 0)}}{\epsilon_n} \sup_{w \in B_{\Db}(t_n;\epsilon_n)} K_{\Db}(F_n(w), w) \leq  \frac{AC_1}{a} \lim_{n \rightarrow \infty} \frac{1}{r_n^4} E(5r_n/4) = 0.
\end{align*}
So Proposition~\ref{prop:quant_id} implies that $F_n$ converges locally uniformly to $\id$. Thus $\pi \circ f \circ \varphi = \id$. 
\end{proof}

\begin{proposition}\label{prop:id} If $\eta \in \partial \Omega$ is a strongly convex point of $\partial \Omega$ and $q_n \in \Omega$ is a sequence with $q_n \rightarrow \eta$, then $f(q_n) \rightarrow \eta$. 
\end{proposition}

\begin{proof} The proposition is obvious if $\eta = \xi_0$. So suppose that $\eta \neq \xi_0$. 

Suppose for a contradiction that $f(q_n)$ does not converge to $\eta$. Then, by passing to a subsequence, we can suppose that $f(q_n) \rightarrow \eta^\prime \in \overline{\Omega}$ where $\eta^\prime \neq \eta$. 

By the previous proposition, for each $q_n$ there exists a complex geodesic $\varphi_n : \Db \rightarrow \Omega$ and a good left inverse $\pi_n : \Omega \rightarrow \Db$ such that $\varphi_n(0)=q_n$, $\pi_n \circ f \circ \varphi_n = \id$, and 
\begin{align*}
\lim_{z \rightarrow 1} d_{\Euc}\left(\varphi_n(z), \partial \Omega \cap H_{\xi_0}\partial \Omega\right) = 0.
\end{align*}

Then 
\begin{align*}
K_\Omega(f \circ \varphi_n(z), f \circ \varphi_n(w)) \leq K_{\Db}(z,w)
\end{align*}
and
\begin{align*}
K_\Omega(f \circ \varphi_n(z), f \circ \varphi_n(w)) \geq K_{\Db}(\pi_n \circ f \circ \varphi_n(z),\pi_n \circ f \circ \varphi_n(w)) = K_{\Db}(z,w).
\end{align*}
So $f \circ \varphi_n$ is also a complex geodesic. 

Since $\eta$ is a strongly convex point, 
\begin{align*}
\partial \Omega \cap H_{\eta} \partial \Omega = \{ \eta\}
\end{align*} 
and so $H_{\eta} \partial \Omega \neq H_{\xi_0} \partial \Omega$. Then by Proposition~\ref{prop:visible} and after possibly passing to a subsequence, there exists $s_n \in \Rb$ such that $\varphi_n \circ a_{s_n}$ converges locally uniformly to a complex geodesic $\varphi : \Db \rightarrow \Omega$. Further
\begin{align*}
\lim_{z \rightarrow 1} d_{\Euc}\left(\varphi(z), \partial \Omega \cap H_{\xi_0} \partial \Omega\right) = 0
\end{align*}
and
\begin{align*}
\lim_{z \rightarrow -1} \varphi(z)= \eta
\end{align*}
since $\partial \Omega \cap H_{\eta} \partial \Omega = \{ \eta\}$. 

The complex geodesics $f \circ \varphi_n \circ a_{s_n}$ converge locally uniformly to $f \circ \varphi$ and 
\begin{align*}
(f \circ \varphi_n \circ a_{s_n})(a_{-s_n}(0))=f(q_n) \rightarrow \eta^\prime,
\end{align*}
 so Proposition~\ref{prop:convergence_at_infinity} implies that
\begin{align*}
\lim_{z \rightarrow -1} d_{\Euc}\left(f(\varphi(z)), \partial \Omega \cap H_{\eta^\prime} \partial \Omega\right) = 0.
\end{align*}

By Montel's theorem and possibly passing to another subsequence we can assume that $a_{-s_n} \circ \pi_n$ converges locally uniformly to some $\pi: \Omega \rightarrow \Db$. Then $\pi$ is a good left inverse of $\varphi$ and $\pi \circ f \circ \varphi = \id$. For each $z \in \Db$ let $H_z$ be the complex hyperplane such that 
\begin{align*}
\pi^{-1}(z) = H_z \cap \Omega.
\end{align*}
Then since 
\begin{align*}
\lim_{z \rightarrow -1} \varphi(z)= \eta,
\end{align*}
Proposition~\ref{prop:inv_hyperplane_convergence} implies that
\begin{align*}
\lim_{z \rightarrow -1} H_z = H_{\eta} \partial \Omega
\end{align*}

If $z \in \Db$, then $\pi( f(\varphi(z))) = z$ and so $f(\varphi(z)) \in H_z$. Thus 
 \begin{align*}
\lim_{z \rightarrow -1} d_{\Euc} \left( f(\varphi(z)), \partial \Omega \cap H_{\eta} \partial \Omega\right) = 0.
\end{align*}
But $\eta \in \partial \Omega$ is a strongly convex point and so $\{\eta \} = \partial \Omega \cap H_{\eta} \partial \Omega$, thus
\begin{align*}
\lim_{z \rightarrow -1} f(\varphi(z))= \eta
\end{align*}
which contradicts the fact that $H_{\eta^\prime}\partial \Omega \neq H_{\eta}\partial \Omega$ and
\begin{align*}
\lim_{z \rightarrow -1} d_{\Euc}\left(f(\varphi(z)), \partial \Omega \cap H_{\eta^\prime} \partial \Omega\right) = 0.
\end{align*}
\end{proof}

\begin{lemma} There exists a strongly convex point $\eta_0 \in \partial \Omega$. \end{lemma}

\begin{proof} Fix a point $z_0 \in \Omega$. Pick $\eta_0 \in \partial \Omega$ such that 
\begin{align*}
\norm{\eta_0 - z_0} = \max\{ \norm{\eta - z_0} : \eta \in \partial \Omega\}
\end{align*}
and let $R = \norm{\eta_0 - z_0}$. Then $\Omega \subset \Bb_d(z_0; R)$ and $\eta_0 \in \partial \Omega \cap \partial \Bb_d(z_0; R)$. So $\eta_0$ is a strongly convex point of $\partial \Omega$. 
\end{proof}

We now claim that $f$ is the identity map. Since $\Omega$ has $C^2$ boundary, there exists a neighborhood $U$ of $\eta_0$ where $\partial \Omega$ is strongly convex at every $\eta \in U \cap \partial \Omega$. 

Fix a point $w_0 \in \Omega$. Consider the complex affine line $L$ containing $w_0$ and $\eta_0$. Then $L \cap \Omega$ is a convex and hence simply connected, so by the Riemann mapping theorem there exists a biholomorphism $\psi: \Db \rightarrow L \cap \Omega$. Since $L \cap \Omega$ is convex, $\partial (L \cap \Omega)$ is a Jordan curve. So by Carath{\'e}odory's extension theorem, $\psi$ extends to a continuous map $\overline{\Db} \rightarrow \overline{L \cap \Omega}$. Next consider the holomorphic map
\begin{align*}
F = (f\circ \psi - \psi): \Db \rightarrow \Cb^d.
\end{align*}
 Since $F$ is bounded, Fatou's Theorem implies that there exists a measurable map $F_\infty : S^1 \rightarrow \Cb^d$ such that 
\begin{align*}
F_\infty\left(e^{i\theta}\right) = \lim_{r \nearrow 1} F\left(re^{i\theta}\right)
\end{align*}
for almost every $e^{i\theta} \in S^1$. However, Proposition~\ref{prop:id} implies that 
\begin{align*}
0 = \lim_{r \nearrow 1} F\left(re^{i\theta}\right)
\end{align*}
when $e^{i\theta} \in V:=\psi^{-1}( U \cap \partial \Omega)$. Since $\eta_0 \in \psi(\overline{\Db})$, $V$ is non-empty  and since $\psi$ is continuous, $V$ is open in $S^1$. So $F_\infty = 0$ on a set of positive measure in $S^1$. So by the Luzin-Privalov Theorem (see~\cite[Chapter 2]{CL1966}), $F \equiv 0$. Thus $f(z) = z$ for all $z \in L \cap \Omega$. In particular, $f(w_0) = w_0$. Since $w_0 \in \Omega$ was arbitrary, we see that $f = \id$.

\section{Proof of Theorem~\ref{thm:main_finite_type}}\label{sec:finite_type}

In this section we describe how to modify the proof of Theorem~\ref{thm:main_convex} to obtain Theorem~\ref{thm:main_finite_type}, but first we recall the definition of the line type of a boundary point. 

Given a function $f: \Cb \rightarrow \Rb$ with $f(0)=0$ let $\nu(f)$ denote the order of vanishing of $f$ at $0$. Suppose that $\Omega \subset \Cb^d$ is a domain and 
\begin{align*}
\Omega = \{ z \in \Cb^d : r(z) < 0\}
\end{align*}
 where $r$ is a $C^\infty$ function with $\nabla r \neq 0$ near $\partial \Omega$. The \emph{line type of a boundary point} $\xi \in \partial \Omega$,  is defined to be
\begin{align*}
\ell(\xi)=\sup \{ \nu( r \circ \psi) | \ \psi : \Cb \rightarrow \Cb^d & \text{ is a non-constant complex affine map }\\
& \text{ with $\psi(0)=\xi$} \}.
\end{align*}
Notice that $\nu(r\circ \psi) \geq 2$ if and only if $\psi(\Cb)$ is tangent to $\Omega$. McNeal~\cite{M1992} proved that if $\Omega$ is convex  then $\xi \in \partial \Omega$ has finite line type if and only if it has finite type in the sense of D'Angelo (also see~\cite{BS1992}). 

For the rest of the section suppose that $\Omega \subset \Cb^d$ is a bounded convex domain with $C^\infty$ boundary, $f : \Omega \rightarrow \Omega$ is holomorphic map, and there exists $\xi_0 \in \partial \Omega$ such that $\ell(\xi_0) < +\infty$ and 
\begin{align*}
f(z) = z + { \rm o}\left( \norm{z-\xi_0}^{4-1/\ell(\xi_0)}\right).
\end{align*}
Then there exists a non-decreasing function $E : [0,\infty) \rightarrow [0,\infty)$ such that 
\begin{align*}
\norm{f(z)-z} \leq E( \norm{z-\xi_0})
\end{align*}
and 
\begin{align*}
\lim_{ r \rightarrow 0} \frac{E(r)}{r^{4-1/\ell(\xi_0)}} = 0.
\end{align*}

The rest of the proof is identical to the proof of Theorem~\ref{thm:main_convex} except that Lemma~\ref{lem:lower_bd_epsilon_n} is replaced with the following stronger result. 

\begin{lemma}There exists some $a > 0$ such that $\epsilon_n \geq a r_n^{1-1/\ell(\xi_0)}$ for all $n$. \end{lemma}

\begin{proof} By~\cite[Corollary 1.7]{AT2002} there exists a neighborhood $U$ of $\xi_0$ and some $\alpha_0 > 0$ such that 
\begin{align*}
k_\Omega(z;v) \geq\alpha_0 \frac{\norm{v}}{\delta_\Omega(z)^{1/\ell(\xi_0)}}
\end{align*}
for all $z \in U \cap \Omega$ and $v \in \Cb^d$. 

Since $p_n \rightarrow \xi_0$ and $r_n \rightarrow 0$, there exists $N > 0$ such that $\Bb_d(p_n;r_n/4) \subset U$ when $n \geq N$. So for $z \in \Bb_d(p_n;r_n/4)$ and $n \geq N$ we have 
\begin{align*}
K_\Omega(z,p_n) \geq \frac{\alpha}{r_n^{1/\ell(\xi_0)}}\norm{z-p_n}
\end{align*}
where $\alpha = (4/5)^{1/\ell(\xi_0)}\alpha_0$. So if $a_0 = 1/(4\alpha)$ and $z \in B_\Omega\left(p_n; a_0r_n^{1-1/\ell(\xi_0)}\right)$ then 
\begin{align*}
\norm{z-p_n} \leq \frac{r_n^{1/\ell(\xi_0)}}{\alpha}K_\Omega(z,p_n) \leq r_n/4.
\end{align*}
So there exists $a > 0$ such that $\epsilon_n \geq a r_n^{1-1/\ell(\xi_0)}$ for all $n$.
\end{proof}

\part{Proof of Theorem~\ref{thm:main_intro_biholo} }

\section{The geometry of the tangent bundle}\label{sec:unit_tangent_bundle}

In this section we recall the definition of the Sasaki metric and give some basic estimates. 

Let $(M,g)$ be a complete Riemannian manifold and let $\pi:TM \rightarrow M$ be the tangent bundle. Define the \emph{vertical subbundle} of  $TTM \rightarrow TM$ by
\begin{align*}
V(X) = \ker d(\pi)_X.
\end{align*}
Next let $\nabla$ be the Levi-Civita connection on $M$. Given $X \in TM$, define the \emph{connection map} $K_X: T_XTM \rightarrow T_{\pi(X)}M$ as follows: given some $\xi \in T_XTM$ let $\sigma:(-\epsilon, \epsilon) \rightarrow TM$ be a curve with $\sigma^\prime(0)=\xi$. Then define
\begin{align*}
K(\xi) = (\nabla_{\alpha^\prime(0)} \sigma)(0)
\end{align*}
where $\alpha = \pi \circ \sigma$ and we view $\sigma$ as a vector field along $\alpha$. This is a well defined linear map (see for instance~\cite[Lemma 1.13]{P1999}). Then define the \emph{horizontal subbundle} of  $TTM \rightarrow TM$ by
\begin{align*}
H(X) =\ker K_X.
\end{align*}
Then for every $X \in TM$ we have 
\begin{align*}
T_XTM = V(X) \oplus H(X)
\end{align*}
and the map 
\begin{align*}
\xi \in T_X TM \rightarrow \Big(d(\pi)_X \xi, K_X(\xi)\Big) \in T_{\pi(X)} M \oplus T_{\pi(X)} M
\end{align*}
is a vector space isomorphism (see for instance~\cite[Lemma 1.15]{P1999}). 

Using the maps defined above we can define a Riemannian metric $h$ on $TM$. Given $X \in TM$ and $\xi \in T_XTM$ define
\begin{align*}
h_X(\xi, \xi) = g_{\pi(X)}\Big( d(\pi)_X \xi,d(\pi)_X \xi\Big) + g_{\pi(X)}\Big( K_X(\xi), K_X(\xi)\Big).
\end{align*}
Then $h$ is a complete Riemannian metric on $TM$ called the \emph{Sasaki metric}.

Let $d_{TM}$ denote the distance on $TM$ induced by $h$. Let 
\begin{align*}
T^1M = \{ X \in TM : \norm{X}_g = 1\}
\end{align*}
denote the unit tangent bundle of $M$ and let $d_{T^1M}$ denote the distance on $T^1M$ induced by restricting $h$ to $T^1M$.

%

We end this section with two estimates. Both are applications of basic methods in Riemannian geometry, but we provide proofs in Appendix~\ref{sec:appendix}. 

\begin{proposition}\label{prop:unit_tangent_vs_tangent}  If $(M,g)$ is a complete Riemannian manifold and $X,Y \in T^1 M$, then 
\begin{align*}
d_{T^1M}(X,Y) \leq (\pi+1) d_{TM}(X,Y).
\end{align*}
\end{proposition}

\begin{proposition}\label{prop:geod_spread} If $(M,g)$ is a complete Riemannian manifold with sectional curvature bounded in absolute value by $\kappa > 0$ and $\gamma_1, \gamma_2 : [0,\infty) \rightarrow M$ are geodesics, then 
\begin{align*}
d_M(\gamma_1(t), \gamma_2(t)) \leq \exp\left( \frac{\kappa+1}{2}t \right)d_{T^1M}(\gamma_1^\prime(0), \gamma_2^\prime(0))
\end{align*}
for $t > 0$. 
\end{proposition}

\section{Two lower bounds}\label{sec:two_lower_bounds}

In this section we establish two lower bounds for metrics with property-$(BG)$.

\begin{proposition}\label{prop:useful_lower_bd_1} Suppose that $\Omega \subset \Cb^d$ is a bounded domain and $g$ is a complete K{\"a}hler metric on $\Omega$ whose sectional curvature is bounded in absolute value by $\kappa > 0$. Then there exists some $a > 0$ such that 
\begin{align*}
a \norm{v} \leq \sqrt{ g_z(v,v)}
\end{align*}
for all $z \in \Omega$ and $v \in \Cb^d$. 
\end{proposition}

\begin{proof} By scaling $\Omega$ we may assume that $\Omega \subset \Bb_d$ where $\Bb_d$ is the unit ball in $\Cb^d$. Let $h$ be the Bergman metric on $\Bb_d$. Then $h$ has  holomorphic bisectional curvature bounded from above by a negative number. Further there exists some $\delta > 0$ such that 
\begin{align*}
\delta \norm{v} \leq \sqrt{ h_z(v,v)}
\end{align*}
for all $z \in \Bb_d$ and $v \in \Cb^d$. Then applying the Yau Schwarz Lemma~\cite{Y1978} to the inclusion map $\Omega \hookrightarrow \Bb_d$ shows that there exists some $C > 0$ such that 
\begin{align*}
C\sqrt{ h_z(v,v)} \leq  \sqrt{ g_z(v,v)}
\end{align*}
for all $z \in \Omega$ and $v \in \Cb^d$. 
\end{proof}

Next we use a result of Cheeger, Gromov, and Taylor to provide a lower bound on the injectivity radius. Suppose $(M,g)$ is a Riemannian manifold. Given $x \in M$ we define \emph{the injectivity radius at $x$} to be
\begin{align*}
{\rm inj}_g(x) = \max\{ R > 0 : \exp_x|_{B_x(s)} \text{ is injective for all } 0 < s < R\}
\end{align*}
where $B_x(s) \subset T_xM$ is the open ball of radius $r$ centered at $0$ in the inner product space $(T_x M, g_x)$. 

\begin{proposition}\label{prop:useful_lower_bd_2}  Suppose that $\Omega \subset \Cb^d$ is a bounded domain and $g$ is a complete Riemannian metric on $\Omega$ such that:
\begin{enumerate}
\item sectional curvature of $g$ is bounded in absolute value by $\kappa > 0$ and
\item there exists $a,A > 0$ such that 
\begin{align*}
a \norm{v} \leq \sqrt{ g_z(v,v)} \leq A \frac{ \norm{v}}{\delta_\Omega(v)}
\end{align*}
for all $z \in \Omega$ and $v \in \Cb^d$. 
\end{enumerate}
 Then there exists some $I_0 > 0$ such that 
\begin{align*}
{ \rm inj}_g(z) \geq I_0 \delta_\Omega(z)^{4d+1}
\end{align*}
for all $z \in \Omega$. 
\end{proposition}

\begin{proof} 
For $z \in \Omega$ and $r > 0$ let $B_g(z,r)$ the open ball of radius $r$ centered at $z$ in $(\Omega, g)$. Then let $V_g(z,r)$ denote the volume of $B_g(z,r)$ in $(\Omega, g)$. For $n \in \Nb$,  $\lambda \in \Rb$, and $r > 0$ let $V_{\lambda}^n(r)$ denote the volume of the ball of radius $r$ in the $n$-dimensional model space $M^n_\lambda$ with constant curvature $\lambda$.  With this notation, Theorem 4.7 in~\cite{CGT1982} implies that
\begin{align}
\label{eq:CGT_est}
{ \rm inj}_g(z) \geq \frac{r}{2} \frac{ V_g(z,r)}{V_g(z,r) + V_{-\kappa}^{2d}(2r)}
\end{align}
for all $r < \pi/ (4 \sqrt{\kappa})$. Finally, fix $V_0 > 0$ such that 
\begin{align*}
V_{-\kappa}^{2d}(2r) \leq V_0
\end{align*}
when $r < 2$.

Fix $z \in \Omega$ sufficiently close to $\partial \Omega$ and let 
\begin{align*}
r = \frac{\delta_\Omega(z)}{2a}.
\end{align*}
Then by the estimates on $g$, 
\begin{align*}
 \left\{ w \in \Omega : \norm{z-w} \leq \frac{1}{4aA}\delta_\Omega(z)^2 \right\}\subset B_g(z,r) \subset \left\{ w \in \Omega : \norm{z-w} \leq \frac{1}{2}\delta_\Omega(z) \right\}.
\end{align*}

We can assume that $r < \min\{ 1, 1/(4 \sqrt{\kappa})\}$. Then
\begin{align}
\label{eq:vol_est_model}
V_{-\kappa}^{2d}(2r) \leq V_0.
\end{align}
Next we estimate $V_g(z,r)$. Let $\Vol_g$ denote the Riemannian volume associated to $g$. By the estimates on $g$, if 
\begin{align*}
E \subset  \left\{ w \in \Omega : \norm{z-w} \leq \frac{1}{2}\delta_\Omega(z) \right\},
\end{align*}
then
\begin{align*}
a^{2d} \lambda(E) \leq \Vol_g(E) \leq \frac{2^{2d}A^{2d}}{\delta_\Omega(z)^{2d}}\lambda(E)
\end{align*}
where $\lambda(E)$ is the Lebesgue measure of $E$. So there exists a constant $A_0 > 1$ such that 
\begin{align}\label{eq:vol_est_manifold}
\frac{1}{A_0} \delta_\Omega(z)^{4d} \leq V_g(z,r) \leq A_0.
\end{align}
Thus by Equations~\eqref{eq:CGT_est}, \eqref{eq:vol_est_model}, and~\eqref{eq:vol_est_manifold}  there exists a constant $I_0 >0$ such that 
\begin{align*}
{ \rm inj}_g(z) \geq I_0 \delta_\Omega(z)^{4d+1}.
\end{align*}
\end{proof}

\section{Deforming metrics}\label{sec:deform}

In this section we recall a result that allows us to deform a Riemannian metric with bounded sectional curvature to obtain a new metric with better properties. 

\begin{theorem}\cite{Shi1989, CZ2006}\label{thm:deform}  Suppose $(M,g)$ is a complete Riemannian manifold whose sectional curvature is bounded in absolute value by $\kappa > 0$. Then for every $\epsilon > 0$ there exists a complete Riemannian metric $\wt{g}$ on $M$ such that:
\begin{enumerate}
\item the sectional curvature of $\wt{g}$ is bounded in absolute value by $\kappa+\epsilon$,
\item the metrics $\wt{g}$ and $g$ are $(1+\epsilon)$-bi-Lipschitz, 
\item if $ \wt{R}$ is the curvature tensor of $\wt{g}$, then 
\begin{align*}
\sup_{x \in M} \abs{ \wt{\nabla}^q \wt{R} }< \infty
\end{align*}
where $ \wt{\nabla}^q$ denotes the $q^{th}$ covariant derivative with respect to $\wt{g}$, and 
\item ${ \rm Isom}(M,g) \leq { \rm Isom}(M, \wt{g})$. 
\end{enumerate}
\end{theorem}

The metric $\wt{g}$ is obtained by considering the Ricci flow starting at $g$:
\begin{align*}
\frac{\partial}{\partial t} g = - {\rm Ric}(g).
\end{align*}
Shi~\cite{Shi1989} proved that there exists some $T > 0$ such that the Ricci flow starting at $g$ has a solution $g_t$ for $t \in [0,T]$ and for any $t \in (0,T]$ the metric $g_t$ satisfies parts (2) and (3). Chen and Zhu~\cite{CZ2006} proved that this solution is unique and hence that ${ \rm Isom}(M,g) \leq { \rm Isom}(M, g_t)$. For precise control over the sectional curvature see for instance~\cite{Kap2005}.

\section{A distance estimate}\label{sec:dist_est}

The main result in this section says that given a complete Riemannian manifold $(M,g)$ and two geodesics $\gamma, \sigma :  I \rightarrow M$ the distance between $\gamma^\prime(0)$ and $\sigma^\prime(0)$ can be estimated from the distance between $\gamma(t)$ and $\sigma(t)$ over a short time interval. Before stating the theorem we need some more notation.

A subset $X$ in a Riemannian manifold $(M,g)$ is said to be \emph{strongly convex} if any two points in $X$ are joined by a unique minimal geodesic and this geodesic is contained in $X$. Given $x \in M$ we define \emph{the convexity radius at $x$} to be
\begin{align*}
r_g(x) = \max\{ R > 0 : B_g(x,s) \text{ is strongly convex for all } 0 < s < R\}
\end{align*}
where $B_g(x,R) \subset M$ is the open ball of radius $R$ centered at $x$. 

The injectivity  radius and convexity radius are related by the following result. 

\begin{theorem}\cite[Proposition 20]{berger2003}\label{thm:inj_vs_convex} Suppose $(M,g)$ is a complete Riemannian manifold with sectional curvature bounded in absolute value by $\kappa > 0$ . If $x \in M$, then 
\begin{align*}
\min\left\{ \frac{\pi}{2 \sqrt{\kappa}}, \frac{1}{2}{ \rm inj}_g(x)\right\} \leq r_g(x) \leq \frac{1}{2}{ \rm inj}_g(x).
\end{align*}
\end{theorem}

We are now ready to state the main result of this section. 

\begin{theorem}\label{thm:dist_est}  Suppose $(M,g)$ is a complete Riemannian manifold and 
\begin{align*}
\sup\{ \abs{\nabla^q R} : x \in M, q=0,1,2\} < \infty
\end{align*}
where $R$ is the curvature tensor of $(M,g)$. Then there exists $A>1$ such that: if $x \in M$, 
\begin{align*}
0 < \epsilon < \min \left\{ {\rm r}_g(x)/2,1\right\},
\end{align*}
  and $\gamma, \sigma :  [0,\epsilon] \rightarrow M$ are unit speed geodesics with $\gamma(0)=x$, then 
\begin{align*}
d_{T^1M}(\gamma^\prime(0), \sigma^\prime(0)) \leq \frac{A}{\epsilon} \max_{t \in [0,\epsilon]} d_M (\gamma(t), \sigma(t)).
\end{align*}
\end{theorem}

To prove the Theorem we will use a result of Eichhorn. Recall, that a chart $(U,\varphi)$ of a Riemamnian manifold $(M,g)$ is a \emph{normal chart centered at $x$ with radius $r$}  if $U = B_g(x,r)$ and $\varphi^{-1} = \exp_x \circ I$ for some linear isometry  $I : \Rb^d \rightarrow (T_x M, g_x)$. 

\begin{theorem}\cite[Corollary 2.6]{E1991}\label{thm:Eichhorn} Suppose $(M,g)$ is a complete Riemannian manifold and 
\begin{align*}
\sup\{ \abs{\nabla^q R} : x \in M, q=0,1,2\} < \infty
\end{align*}
where $R$ is the curvature tensor of $(M,g)$. If $r_0 > 0$, then there exists $\wt{C}>0$ such that: if $x \in M$, $(U, \varphi)$ is a normal chart centered at $x$ of radius at most $r_0$, and $h = \varphi_*g$, then 
\begin{align*}
\sup_{\varphi(U)} \abs{ \frac{\partial^{\abs{\alpha}} h_{i,j}}{\partial u^{\alpha}}} \leq \wt{C}
\end{align*}
for every multi-index $\alpha$ with $\abs{\alpha} \leq 2$. 
\end{theorem}

\subsection{Proof of Theorem~\ref{thm:dist_est}}

For the rest of the section let $(M,g)$ be a complete Riemannian manifold with
\begin{align*}
\sup\{ \abs{\nabla^q R} : x \in M, q=0,1,2\} < \infty
\end{align*}
where $R$ is the curvature tensor of $(M,g)$. Let $\wt{C} > 0$ be the constant from Theorem~\ref{thm:Eichhorn} with $r_0 =1$. 

\begin{lemma}\label{lem:norm_chart_1} There exists constants $r_1,A_1 > 0$ such that: if $x \in M$, $(U, \varphi)$ is a normal chart centered at $x$ of radius at most $r_1$, and $\gamma : [0,T] \rightarrow M$ is a unit speed geodesic with image in $U$, then
\begin{align*}
\frac{1}{A_1} \leq \norm{(\varphi \circ \gamma)^\prime(t)} \leq A_1.
\end{align*}
In particular, if $p,q \in U \cap B_g(x, { \rm r}_g(x))$ then 
\begin{align*}
\frac{1}{A_1} \norm{\varphi(p)-\varphi(q)} \leq d_M(p, q) \leq A_1 \norm{\varphi(p)-\varphi(q)}.
\end{align*}
\end{lemma}

\begin{proof} Let 
\begin{align*}
r_1 = \min\left\{1,\frac{1}{2d\wt{C}}\right\}.
\end{align*}
Then suppose that $(U, \varphi)$ is a normal chart centered at $x$ of radius at most $r_1$. Let $\wt{\gamma} = \varphi \circ \gamma$ and $h = \varphi_*g$. Then $\norm{\wt{\gamma}^\prime(t)}_h = 1$. 

Since $h$ at $u=0$ is the standard Euclidean inner product, we see that 
\begin{align*}
\abs{\norm{v}_h^2-\norm{v}^2} \leq \wt{C}r_1\sum_{i,j} \abs{v_i v_j} \leq \frac{1}{2} \wt{C}r_1\sum_{i,j} \abs{v_i}^2 + \abs{v_j}^2 = \wt{C}r_1 d\norm{v}^2 \leq \frac{1}{2} \norm{v}^2.
\end{align*}
So
\begin{align*}
\frac{1}{2} \norm{v} \leq \norm{v}_h \leq 2\norm{v}
\end{align*}
and so
\begin{align*}
\frac{1}{2} \leq \norm{\wt{\gamma}^\prime(t)} \leq 2.
\end{align*}

Next suppose that $p,q \in U \cap B_g(x, { \rm r}_g(x))$. Then let $\sigma:[0,T] \rightarrow M$ be a unit speed geodesic joining $p$ to $q$. Then the image of $\sigma$ is contained in $U$ so 
\begin{align*}
d_M(p, q) 
&=  \int_0^T \norm{\sigma^\prime(t)}_g dt = \int_0^T \norm{(\varphi \circ \sigma)^\prime(t)}_h dt \\
& \geq \frac{1}{2}\int_0^T \norm{(\varphi \circ \sigma)^\prime(t)} dt \geq \frac{1}{2}\norm{\varphi(p)-\varphi(q)}.
\end{align*}
On the other hand, if $f(t) = t\varphi(p)+(1-t)\varphi(q)$ , then 
\begin{align*}
d_M(p, q) 
&\leq \int_0^1 \norm{(\varphi^{-1} \circ f)^\prime(t)}_g dt = \int_0^1 \norm{f^\prime(t)}_h dt\\
&  \leq 2\int_0^1\norm{\varphi(p)-\varphi(q)} dt = 2 \norm{\varphi(p)-\varphi(q)}.
\end{align*}

\end{proof}

\begin{lemma}\label{lem:CS_bd} There exist a constant $\wt{C}_1 > 0$ such that: if $x \in M$, $(U, \varphi)$ is a normal chart centered at $x$ of radius at most $r_1$, $h=\varphi_* g$, and 
\begin{align*}
\Gamma_{ij}^k = \frac{1}{2} \sum_k \left(\frac{\partial h_{jk}}{\partial u_i}+\frac{\partial h_{ki}}{\partial u_j}-\frac{\partial h_{ij}}{\partial u_k}\right)h^{km},
\end{align*}
then 
\begin{align*}
\max\left\{ \abs{\Gamma_{ij}^k}, \abs{ \frac{\partial\Gamma_{ij}^k}{\partial u_1}}, \dots, \abs{ \frac{\partial \Gamma_{ij}^k}{\partial u_d}} \right\}\leq \wt{C}_1
\end{align*}
on $\varphi(U)$.
\end{lemma}

\begin{proof}
The proof of the last lemma provides a uniform bound on $h^{km}$. So the Lemma follows from Theorem~\ref{thm:Eichhorn}.
\end{proof}

\begin{lemma}\label{lem:norm_chart_2}  There exist a constant $A_2 > 0$ such that: if $x \in M$, $(U, \varphi)$ is a normal chart centered at $x$ of radius at most $r_1$, and $\gamma, \sigma : [0,T] \rightarrow M$ are unit speed geodesics with images in $U$, then
\begin{align*}
\norm{F^{\prime\prime}(t) } \leq A_2 \Big( \norm{F(t)}+\norm{F^\prime(t)} \Big)
\end{align*}
where $F(t) = (\varphi \circ \gamma)(t)- (\varphi \circ \sigma)(t)$. 
\end{lemma}

\begin{proof}
Let $\wt{\gamma} = \varphi \circ \gamma$, $\wt{\sigma} = \varphi \circ \sigma$, and $h = \varphi_*g$. By~\cite[page 62]{dC1992}, the components of $\wt{\gamma}^{\prime\prime}- \wt{\sigma}^{\prime\prime}$ satisfy the differential equation 
\begin{align*}
\wt{\gamma}^{\prime\prime}_k(t)-\wt{\sigma}^{\prime\prime}_k(t) =\sum_{i,j}\wt{\sigma}^{\prime}_i(t)\wt{\sigma}^{\prime}_j(t) \Gamma_{ij}^k(\wt{\sigma}(t))- \sum_{i,j}\wt{\gamma}^{\prime}_i(t)\wt{\gamma}^{\prime}_j(t) \Gamma_{ij}^k(\wt{\gamma}(t)).
\end{align*}
 
By Lemma~\ref{lem:norm_chart_1}
\begin{align*}
\max\{ \norm{\wt{\gamma}^\prime(t)}, \norm{\wt{\sigma}^\prime(t)}\} \leq A_1.
\end{align*}
Then since the function 
\begin{align*}
(u,X) \in U \times \Rb^d \rightarrow \sum_{i,j} X_iX_j\Gamma_{ij}^k(u)
\end{align*}
has locally bounded first derivatives, there exists some $\wt{A}_2 > 0$ such that 
\begin{align*}
\norm{\wt{\gamma}^{\prime\prime}_k(t)-\wt{\sigma}^{\prime\prime}_k(t)} \leq \wt{A}_2 \Big( \norm{\wt{\gamma}(t)-\wt{\sigma}(t)}+  \norm{\wt{\gamma}^\prime(t)-\wt{\sigma}^\prime(t)}\Big).
\end{align*}

\end{proof}

\begin{lemma}\label{lem:norm_chart_3}  There exist a constant $A_3 > 0$ such that: if $x \in M$, $(U, \varphi)$ is a normal chart centered at $x$ of radius at most $r_1$, then 
\begin{align*}
d_{T^1M}\Big( (\varphi^{-1}(u_1), & d(\varphi^{-1})_{u_1}X), (\varphi^{-1}(u_2), d(\varphi^{-1})_{u_2}Y)\Big)\\
&  \leq A_3\max\{1, \norm{X}, \norm{Y}\} \Big( \norm{u_1-u_2} + \norm{X-Y} \Big)
\end{align*}
for all $u_1,u_2 \in \varphi(U)$ and $X,Y \in \Rb^d$.
\end{lemma}

\begin{proof} In the local coordinates $(u_1,\dots, u_d, X_1, \dots, X_d) \in U \times \Rb^d$ the Sasaki metric is given by 
\begin{align*}
h_{i,j} du^i du^j + h_{i,j} DX^i DX^j
\end{align*}
where 
\begin{align*}
DX^i = dX^i + \Gamma_{jk}^i X_i du^k.
\end{align*}
So the estimate follows form Theorem~\ref{thm:Eichhorn} and Lemma~\ref{lem:CS_bd}.
\end{proof}

We will also use the following simple observation:

\begin{lemma}\label{lem:linear_alg} If $X,Y \in \Rb^d$ and $\epsilon \in (0,2)$, then 
\begin{align*}
\max_{t \in [0,\epsilon] } \norm{X+tY} \geq \frac{\epsilon}{4} \left(\norm{X}+\norm{Y}\right).
\end{align*}
\end{lemma}

\begin{proof}
If $\norm{X} \geq \frac{\epsilon}{2}\norm{Y}$, then 
\begin{align*}
\norm{X} \geq \frac{1}{2} \norm{X}+\frac{\epsilon}{4}\norm{Y} \geq \frac{\epsilon}{4} \left(\norm{X}+\norm{Y}\right).
\end{align*}
If $\norm{X} \leq \frac{\epsilon}{2}\norm{Y}$, then 
\begin{align*}
\norm{X+\epsilon Y} \geq \epsilon \norm{Y} - \norm{X} \geq \frac{\epsilon}{2} \norm{Y}\geq \frac{\epsilon}{4} \left(\norm{X}+\norm{Y}\right).
\end{align*}
\end{proof}

We now prove Theorem~\ref{thm:dist_est} in a special case. Let 
\begin{align*}
r_2 = \min\left\{ r_1, \frac{1}{8\sqrt{d}A_2} \right\}.
\end{align*}

\begin{lemma}\label{lem:special_case}
There exists $A_4>1$ such that: if $x \in M$, 
\begin{align*}
0 < \epsilon <  \min\left\{{ \rm r}_g(x)/2, r_2\right\},
\end{align*}
  and $\gamma, \sigma :  [0,\epsilon] \rightarrow M$ are unit speed geodesics with $\gamma(0)=x$, then 
\begin{align*}
d_{T^1M}(\gamma^\prime(0), \sigma^\prime(0)) \leq \frac{A_4}{\epsilon} \max_{t \in [0,\epsilon]} d_M (\gamma(t), \sigma(t)).
\end{align*}
\end{lemma}

\begin{proof} The proof is divided into two cases:

\noindent  \textbf{Case 1:} Suppose that 
  \begin{align*}
  \epsilon \leq \max_{t \in [0,\epsilon]} d_M (\gamma(t), \sigma(t)).
  \end{align*}
 Then by Lemma~\ref{lem:obs3}
\begin{align*}
d_{T^1M}(\gamma^\prime(0), \sigma^\prime(0))
& \leq \pi+ d_M (\gamma(0), \sigma(0)) \\
& \leq \frac{\pi+1}{\epsilon} \max_{t \in [0,\epsilon]} d_M (\gamma(t), \sigma(t)).
\end{align*}

\noindent \textbf{Case 2:} Suppose that 
 \begin{align*}
  \epsilon \geq \max_{t \in [0,\epsilon]} d_M (\gamma(t), \sigma(t)).
  \end{align*}
  Fix $(U, \varphi)$ is a normal chart centered at $x$ with radius $\min\{r_g(x),2r_2\}$. Let $\wt{\gamma} = \varphi \circ \gamma$, $\wt{\sigma}=\varphi \circ \gamma$, and $F= \wt{\gamma}-\wt{\sigma}$.

 Define 
  \begin{align*}
  D = \max_{t \in [0,\epsilon]} \norm{F(t)} + \norm{F^\prime(t)} 
  \end{align*}
  and pick some $t_0 \in [0,\epsilon]$ realizing this maximum. \\
  
  \noindent \textbf{Claim:} For $t \in [0,\epsilon]$
   \begin{align*}
   \norm{F(t_0)+F^\prime(t_0)(t-t_0)} \leq  \norm{F(t)} + \frac{\epsilon}{16}D
   \end{align*}

   \begin{proof}[Proof of Claim:] Let $F=(F_1,\dots, F_d)$. Then by Taylor's theorem 
   \begin{align*}
  F_k(t) =F_k(t_0)+F_k^\prime(t_0)(t-t_0) + \frac{1}{2}F_k^{\prime\prime}(\zeta_k)(t-t_0)^2
   \end{align*}
   for some $\zeta_k$ between $t$ and $t_0$. Further 
   \begin{align*}
    \abs{F_k^{\prime\prime}(\zeta_k)} \leq A_2\Big(\norm{F(\zeta_k)} + \norm{F^\prime(\zeta_k)} \Big) \leq A_2 D 
    \end{align*} 
    by Lemma~\ref{lem:norm_chart_2}. So
   \begin{align*}
   \norm{F(t_0)+F^\prime(t_0)(t-t_0)} \leq  \norm{F(t)} + \frac{\sqrt{d}}{2} A_2D (t-t_0)^2. 
   \end{align*}
   Then 
      \begin{align*}
   \norm{F(t_0)+F^\prime(t_0)(t-t_0)} \leq  \norm{F(t)} + \frac{\epsilon}{16}D
   \end{align*}
   since $\epsilon < (8\sqrt{d}A_2)^{-1}$. 
   \end{proof}
   
     \noindent \textbf{Claim:}    \begin{align*}
   D \leq \frac{16}{\epsilon}\max_{t \in [0,\epsilon]}\norm{F(t)}.
   \end{align*}
   
   \begin{proof}[Proof of Claim:] By Lemma~\ref{lem:linear_alg}
   \begin{align*}
   D \leq \frac{8}{\epsilon}\max_{t \in [0,\epsilon]}\norm{F(t_0)+F^\prime(t_0)(t-t_0)}.
   \end{align*}
   Then by the previous claim
   \begin{align*}
   D \leq \frac{8}{\epsilon}\max_{t \in [0,\epsilon]}\Big(\norm{F(t)} + \frac{\epsilon}{16}D \Big).
   \end{align*}
   So
   \begin{align*}
   D \leq \frac{16}{\epsilon}\max_{t \in [0,\epsilon]}\norm{F(t)}.
   \end{align*}
   
   \end{proof}
   
    By Lemmas~\ref{lem:norm_chart_1}  and \ref{lem:norm_chart_3} 
  \begin{align*}
  d_{T^1M}\left(\gamma^\prime(0), \sigma^\prime(0)\right) \leq A_1A_3 \Big( \norm{F(0)} + \norm{F^\prime(0)} \Big) \leq A_1A_3 D
  \end{align*}
  and by Lemma~\ref{lem:norm_chart_1} 
    \begin{align*}
  d_M(\gamma(t), \sigma(t)) \geq \frac{1}{A_1}  \norm{F(t)}.
  \end{align*}
  So by the previous claim
      \begin{align*}
  d_{T^1M}\left(\gamma^\prime(0), \sigma^\prime(0)\right)  \leq \frac{16A_1^2A_3}{\epsilon}  \max_{t \in [0,\epsilon]} d_M(\gamma(t), \sigma(t)).
  \end{align*}

Thus $A_4 = \max\{ 16A_1^2 A_3, \pi+1\}$ satisfies the statement of the lemma. 

\end{proof}

\begin{proof}[Proof of Theorem~\ref{thm:dist_est}] Suppose $x \in M$, 
\begin{align*}
0 < \epsilon < \min\left\{{ \rm r}_g(x)/2, 1\right\},
\end{align*}
  and $\gamma, \sigma :  [0,\epsilon] \rightarrow M$ are geodesics with $\gamma(0)=x$. If $\epsilon< r_2$, then by Lemma~\ref{lem:special_case}
        \begin{align*}
  d_{T^1M}\left(\gamma^\prime(0), \sigma^\prime(0)\right)  \leq \frac{A_4}{\epsilon}  \max_{t \in [0,\epsilon]} d_M(\gamma(t), \sigma(t)).
  \end{align*}
  If $\epsilon > r_2$, then by Lemma~\ref{lem:special_case}
      \begin{align*}
  d_{T^1M}\left(\gamma^\prime(0), \sigma^\prime(0)\right)  \leq \frac{A_4}{r_2}  \max_{t \in [0,r_2]} d_M(\gamma(t), \sigma(t)) \leq  \frac{A_4r_2^{-1}}{\epsilon}  \max_{t \in [0,\epsilon]} d_M(\gamma(t), \sigma(t)).
  \end{align*}
  So $A = A_4r_2^{-1}$   satisfies the statement of the theorem. 
  
\end{proof}

\section{Proof of Theorem~\ref{thm:main_intro_biholo}}

In this section we prove the following strengthening of Theorem~\ref{thm:main_intro_biholo}.

\begin{theorem}\label{thm:main} Suppose $\Omega \subset \Cb^d$ is a bounded domain,  $\varphi \in \Aut(\Omega)$, and  $\partial \Omega$ satisfies an interior cone condition at $\xi_0 \in \partial \Omega$ with parameters $r,\theta$. Assume there exists an $\varphi$-invariant complete Riemannian metric $g$ on $\Omega$ such that 
\begin{enumerate}
\item the sectional curvature of $g$ is bounded in absolute value by $\kappa > 0$ and
\item there exists $a,A > 0$ such that 
\begin{align*}
a \norm{v} \leq \sqrt{ g_z(v,v)} \leq A \frac{ \norm{v}}{\delta_\Omega(v)}
\end{align*}
for all $z \in \Omega$ and $v \in \Cb^d$. 
\end{enumerate}
If 
\begin{align*}
L > 4d+2+\frac{\sqrt{\kappa}A}{\sin(\theta)}
\end{align*}
and
\begin{align*}
 \varphi(z) = z + { \rm O} \left( \norm{z-\xi_0}^L \right),
\end{align*}
then $\varphi = \id$.
\end{theorem}

\begin{remark} Notice that Theorem~\ref{thm:main} and Proposition~\ref{prop:useful_lower_bd_1} imply Theorem~\ref{thm:main_intro_biholo}. \end{remark}

For the rest of the section suppose that $\Omega$,  $\varphi$, $g$, $\xi_0$, $r$, $\theta$, $\kappa$, $a$, and $A$ satisfy the hypothesis of Theorem~\ref{thm:main}.  Then there exists some $v \in \Cb^d$ such that $\norm{v}=1$ and
\begin{align*}
\Cc(\xi_0,v,\theta,r) \subset \Omega.
\end{align*}
By replacing $\Omega$ with $\frac{1}{2r} \Omega$ and $g$ with $\Phi_*g$ where $\Phi(z) = \frac{1}{2r}z$, we can assume that $r=2$. Notice that this does not change $\theta$, $\kappa$, or $A$. Then 
\begin{align*}
\delta_\Omega(\xi_0+tv) \geq \sin(\theta) t
\end{align*}
for every $t \in (0,1]$.

If we replace $g$ with $\lambda g$ where $\lambda > 0$, then $A$ is replaced by $\sqrt{\lambda}A$ and $\kappa$ is replaced by $\kappa/\lambda$. Thus the quantity 
\begin{align*}
4d+2+\frac{\sqrt{\kappa}A}{\sin(\theta)}
\end{align*}
is invariant under scaling $g$. So we may assume that $\kappa = 1$. 

Suppose that 
\begin{align*}
L >  4d+2+\frac{A}{\sin(\theta)}
\end{align*}
and
\begin{align*}
 \varphi(z) = z + { \rm O} \left( \norm{z-\xi}^L \right).
\end{align*}

Fix $\epsilon > 0$ such that 
\begin{align*}
L >  4d+2+(2+\epsilon)(1+\epsilon)\frac{A}{2\sin(\theta)}.
\end{align*}
Then by Theorem~\ref{thm:deform} we can find a complete Riemannian metric $\wt{g}$ on $\Omega$ such that:
\begin{enumerate}
\item the Riemannian sectional curvature of $\wt{g}$ is bounded in absolute value by $1+\epsilon$,
\item the metrics $\wt{g}$ and $g$ are $(1+\epsilon)$-bi-Lipschitz, 
\item if $ \wt{R}$ is the curvature tensor of $\wt{g}$, then 
\begin{align*}
\sup_{x \in M} \abs{ \wt{\nabla}^q \wt{R} }< \infty
\end{align*}
where $ \wt{\nabla}^q$ denotes the $q^{th}$ covariant derivative with respect to $\wt{g}$, and 
\item $\varphi \in { \rm Isom}(M, \wt{g})$. 
\end{enumerate}
Let $d_\Omega$ be the distance on $\Omega$ induced by $\wt{g}$. 

Next fix a sequence $r_n \in (0, 1]$ with $r_0=1$ and $r_n \rightarrow 0$. Let $p_n=\xi_0+r_n v \in \Omega$. 

\begin{lemma} With the notation above, 
\begin{align*}
d_\Omega(p_n, p_0) \leq \frac{(1+\epsilon)A}{\sin(\theta)} \log \frac{1}{r_n}.
\end{align*}
for every $n \geq 0$. 
\end{lemma}

\begin{proof}
Let $\sigma: [0,1) \rightarrow \Omega$ be the curve $\sigma(t) = \xi_0+(1-t)v$. Then using the fact that 
\begin{align*}
\delta_\Omega(\xi_0+tv) \geq \sin(\theta) t
\end{align*}
for every $t \in (0,1]$, we have
\begin{align*}
d_\Omega(p_n,p_0) 
&\leq \int_0^{1-r_n} \sqrt{\wt{g}_{\sigma(t)}(\sigma^\prime(t), \sigma^\prime(t))} dt \leq \int_0^{1-r_n} \frac{ (1+\epsilon)A \norm{\sigma^\prime(t)}}{\delta_\Omega(\sigma(t))} dt \\
& \leq \frac{(1+\epsilon)A}{\sin(\theta)} \int_0^{1-r_n} \frac{dt}{1-t}  = \frac{(1+\epsilon)A}{\sin(\theta)} \log \frac{1}{r_n}.
\end{align*}
\end{proof}

Next fix some $z_0 \in \Omega$ and let $\gamma_n:[0,T_n] \rightarrow \Omega$ be a unit speed geodesic in $(\Omega, d_\Omega)$ with $\gamma_n(0)=p_n$ and $\gamma_n(T_n)=z_0$. Then 
\begin{align}
\label{eq:T_estimate}
T_n = d_\Omega(z_0,p_n) \leq d_\Omega(z_0,p_0) +\frac{(1+\epsilon)A}{\sin(\theta)} \log \frac{1}{r_n}.
\end{align}

Next let
\begin{align*}
\tau_n = \max\left\{ \tau \in [0,T_n] :\norm{ \gamma_n(t) -p_n} \leq \frac{\sin(\theta) r_n}{4} \text{ for all } t \in [0,\tau]\right\}.
\end{align*}

\begin{lemma}\label{lem:lower_bd_tau} With the notation above, there exists $\delta> 0$ such that
\begin{align*}
\tau_n \geq \delta r_n
\end{align*}
 for all $n$ sufficiently large. \end{lemma}

\begin{proof} Since
\begin{align*}
\sqrt{\wt{g}_z(v,v)} \geq \frac{a}{1+\epsilon} \norm{v}
\end{align*}
for all $z \in \Omega$ and $v \in \Cb^d$, we have
\begin{align*}
d_\Omega(z,w) \geq \frac{a}{1+\epsilon} \norm{z-w}
\end{align*}
for all $z,w \in \Omega$. Now if $\norm{p_n-z_0} > \sin(\theta)  r_n/4$ then 
\begin{align*}
\tau_n = d_\Omega(\gamma_n(0), \gamma_n(\tau_n)) \geq \frac{a}{1+\epsilon} \norm{p_n-\gamma_n(\tau_n)} = \frac{\sin(\theta)  a}{4(1+\epsilon)}r_n.
\end{align*}

\end{proof}

Now pick $\alpha > 0$ such that 
\begin{align*}
\norm{\varphi(z)-z} \leq \alpha \norm{z-\xi_0}^L
\end{align*}
for all $z \in \Omega$. 

\begin{lemma}\label{lem:dist_est_t} There exists $C_1 > 0$ and $N > 0$ such that 
\begin{align*}
d_\Omega\Big( \gamma_n(t), \varphi(\gamma_n(t))\Big) \leq C_1 r_n^{L-1}
\end{align*}
for all $n \geq N$ and $t \in[0, \tau_n]$. 
\end{lemma}

\begin{proof} If $t \in [0,\tau_n]$, then 
\begin{align*}
\norm{\gamma_n(t) - \xi_0} \leq \norm{\gamma_n(t) - p_n}+\norm{p_n - \xi_0}\leq\frac{\sin(\theta)+4}{4} r_n .
\end{align*}
So 
\begin{align*}
\norm{\gamma_n(t) - \varphi(\gamma_n(t))} \leq \alpha\left(\frac{\sin(\theta)+4}{4} r_n \right)^L.
\end{align*}
Since $r_n \rightarrow 0$ and $L>1$ we can pick $N \geq 0$ such that 
\begin{align*}
\norm{\gamma_n(t) - \varphi(\gamma_n(t))} \leq \frac{\sin(\theta)}{4}r_n
\end{align*}
for all $n \geq N$. 

Next define $\sigma_n: [0,1] \rightarrow \Cb^d$ by 
\begin{align*}
\sigma_n(s) = (1-s)\gamma_n(t) + s\varphi(\gamma_n(t)).
\end{align*}
Then for $n \geq N$ we have 
\begin{align*}
\delta_\Omega(\sigma_n(s)) 
&\geq \delta_\Omega(p_n) - \norm{p_n-\sigma_n(s)}  \geq \delta_\Omega(p_n) - \norm{p_n-\gamma_n(t)}-\norm{\gamma_n(t) - \varphi(\gamma_n(t))} \\
& = \frac{\sin(\theta)}{2}r_n.
\end{align*}
So
\begin{align*}
d_\Omega(\gamma_n(t), &\varphi(\gamma_n(t))) 
\leq \int_0^{1} \sqrt{\wt{g}_{\sigma(s)}(\sigma^\prime(s), \sigma^\prime(s))} ds   \leq (1+\epsilon)A\int_0^1 \frac{\norm{\sigma^\prime_n(s)}}{\delta_\Omega(\sigma_n(s)) } ds & \\
&  \leq (1+\epsilon)A\alpha\left(\frac{\sin(\theta)+4}{4} \right)^L \frac{2}{\sin(\theta)}r_n^{L-1}.
\end{align*}

%
%
%

\end{proof}

\begin{lemma}\label{lem:initial_cond_bd} There exists $C_2 > 0$ such that 
\begin{align*}
d_{T^1\Omega}\Big( \gamma_n^\prime(0), (\varphi \circ \gamma_n)^\prime(0) \Big) \leq C_2 r_n^{L-4d-2}
\end{align*}
for all $n \geq N$. 
\end{lemma}

\begin{proof}
Let 
\begin{align*}
\epsilon_n = \min\left\{ {\rm  r}_{\wt{g}}(p_n)/2, \tau_n,1\right\}.
\end{align*}
By Proposition~\ref{prop:useful_lower_bd_2}, Theorem~\ref{thm:inj_vs_convex}, and Lemma~\ref{lem:lower_bd_tau} there exists $E_0 > 0$ such that 
\begin{align*}
\epsilon_n \geq E_0 r_n^{4d+1}.
\end{align*}
By Theorem~\ref{thm:dist_est} there exists some $\beta > 0$ such that 
\begin{align*}
d_{T^1\Omega}\Big( \gamma_n^\prime(0), (\varphi \circ \gamma_n)^\prime(0) \Big) \leq \frac{\beta}{\epsilon_n} \max_{t \in [0,\epsilon_n]} d_\Omega( \gamma_n(t), \varphi(\gamma_n(t)) )
\end{align*}
So by Lemma~\ref{lem:dist_est_t} 
\begin{align*}
d_{T^1\Omega}\Big( \gamma_n^\prime(0), (\varphi \circ \gamma_n)^\prime(0) \Big) \leq  \frac{\beta C_1}{E_0} r_n^{L-4d-2}.
\end{align*}

\end{proof}

Then by Proposition~\ref{prop:geod_spread} and Equation~\eqref{eq:T_estimate}
\begin{align*} 
d_\Omega(z_0, \varphi(z_0)) 
&\leq \exp\left( \frac{1+\epsilon+1}{2}T_n \right) d_{T^1\Omega} \Big( \gamma_n^\prime(0), (\varphi \circ \gamma_n)^\prime(0) \Big) \\
& \leq C_2 \exp\left( \frac{2+\epsilon}{2}d_\Omega(p_0,z_0) \right) r_n^{L-4d-2-(2+\epsilon)(1+\epsilon)\frac{A}{2\sin(\theta)}}.
\end{align*}

Since $r_n \rightarrow 0$ and
\begin{align*}
L > 4d+2+(2+\epsilon)(1+\epsilon)\frac{A}{2\sin(\theta)}
\end{align*}
 we see that $d_\Omega(z_0, \varphi(z_0)) =0$. Hence $\varphi(z_0) = z_0$. Since $z_0$ was arbitrary we then see that $\varphi = \id$. 
 
 
 \begin{remark} In the special case when 
 \begin{align*}
 \inf_{z \in \Omega} { \rm inj}_{g}(z) > 0
 \end{align*}
 it suffices to assume that 
\begin{align*}
L >  2+\frac{\sqrt{\kappa} A}{\sin(\theta)}.
\end{align*}
In this case one first shows that 
 \begin{align*}
 \inf_{z \in \Omega} { \rm inj}_{\wt{g}}(z) > 0.
 \end{align*}
Then Theorem~\ref{thm:inj_vs_convex} implies that
  \begin{align*}
 \inf_{z \in \Omega} { \rm r}_{\wt{g}}(z) > 0.
 \end{align*}
 So in the proof of Lemma~\ref{lem:initial_cond_bd} we can assume $\epsilon_n \geq E_0r_n$  which implies that
 \begin{align*}
d_{T^1\Omega}\Big( \gamma_n^\prime(0), (\varphi \circ \gamma_n)^\prime(0) \Big) \leq C_2 r_n^{L-2}.
\end{align*}
The rest of the argument is identical. 
 \end{remark}

\section{Examples}\label{sec:examples}

Given a domain $\Omega$ let $k_\Omega$ denote the infinitesimal Kobayashi metric on $\Omega$. By the definition of the Kobayashi metric 
\begin{align*}
k_\Omega(z;v) \leq \frac{ \norm{v}}{\delta_\Omega(z)}
\end{align*}
for all $z \in \Omega$ and $v \in \Cb^d$. 

\subsection{HRR domains} Given a bounded domain $\Omega \subset \Cb^d$ let $s_\Omega : \Omega \rightarrow (0,1]$ be the \emph{squeezing function on $\Omega$}, that is 
\begin{align*}
s_\Omega(z) = \sup\{ r : & \text{ there exists an one-to-one holomorphic map } \\
& f: \Omega \rightarrow \Bb_d \text{ with } f(z)=0 \text{ and } r\Bb_d \subset f(\Omega) \}.
\end{align*}
Then define 
\begin{align*}
s(\Omega) = \inf_{z \in \Omega} s_\Omega(z). 
\end{align*}
Notice that $\Omega$ is a HRR domain if and only if $s(\Omega)>0$. Then Sai-Kee Yeung proved the following. 

\begin{theorem}\label{thm:SKYeung}\cite[Theorem 2]{Y2009}  For every $s \in (0,1]$ and $d \in \Nb$, there exists $\alpha=\alpha(s,d)>1$ and $\kappa=\kappa(s,d) > 0$ such that: if $\Omega \subset \Cb^d$ is a bounded HRR domain with $s(\Omega) \geq s$, then 
\begin{enumerate}
\item $k_\Omega$ and $g_\Omega$ are $\alpha$-bi-Lipschitz, and
\item  the sectional curvature of $g_\Omega$ is bounded in absolute value by $\kappa$.
\end{enumerate}
\end{theorem}

As a corollary we have the following. 

\begin{corollary}\label{cor:HRR} If $\Omega \subset \Cb^d$ is a bounded HRR domain, then the K{\"a}hler-Einstein metric has property-$(BG)$. Moreover, we can choose the $\kappa$ and $A$ in the definition of property-$(BG)$ to depend only on $s(\Omega)$ and $d$. 
\end{corollary}

\begin{proof} Let $\alpha$ and $\kappa$ be the numbers from Theorem~\ref{thm:SKYeung}. By definition the Kobayashi metric satisfies 
\begin{align*}
k_\Omega(z;v) \leq \frac{ \norm{v}}{\delta_\Omega(z)}
\end{align*}
and so 
\begin{align*}
\sqrt{g_{\Omega,z}(v,v)} \leq \alpha k_\Omega(z;v) \leq \frac{ \alpha \norm{v}}{\delta_\Omega(z)}.
\end{align*}
\end{proof}

\subsection{Pinched negative curvature} Wu and Yau proved the following. 

\begin{theorem}\label{thm:Wu_Yau}\cite[Theorem 2, Theorem 3]{WY2017}
For every $a, b > 0$ and $d \in \Nb$, there exists $\alpha=\alpha(a,b,d)>1$ and $\kappa=\kappa(a,b,d) > 0$ such that: if $\Omega \subset \Cb^d$ is a bounded domain and there exists a complete K{\"a}hler metric $g$ on $\Omega$ such that 
\begin{align*}
-a \leq H(g) \leq -b,
\end{align*}
then
\begin{enumerate}
\item $\Omega$ has a unique K{\"a}hler-Einstein metric $g_\Omega$ with Ricci curvature $-1$,
\item $k_\Omega$ and $g_\Omega$ are $\alpha$-bi-Lipschitz, and
\item  the sectional curvature of $g_\Omega$ is bounded in absolute value by $\kappa$.
\end{enumerate}
\end{theorem}

Then arguing as in Corollary~\ref{cor:HRR} we have the following. 

\begin{corollary} If $\Omega \subset \Cb^d$ is a bounded domain and there exists a complete K{\"a}hler metric $g$ on $\Omega$ such that 
\begin{align*}
-a \leq H(g) \leq -b
\end{align*}
for some $a > b > 0$, then the K{\"a}hler-Einstein metric has property-$(BG)$. Moreover, we can choose the $\kappa$ and $A$ in the definition of property-$(BG)$ to depend only on $a$, $b$, and $d$. 
\end{corollary}

\appendix

\section{Some proofs}\label{sec:appendix}

In this section we prove prove Propositions~\ref{prop:unit_tangent_vs_tangent} and~\ref{prop:geod_spread}. In this section, if $(M,g)$ is a complete Riemannian manifold and $\sigma : [a,b] \rightarrow M$ is a smooth curve let $\ell_g(\sigma)$ denote the length of $\sigma$ in $(M,g)$. 

\subsection{Proof of Proposition~\ref{prop:unit_tangent_vs_tangent}}

For the rest of the subsection suppose that $(M,g)$ is a complete Riemannian manifold. Before proving the proposition we need three lemmas. 

\begin{lemma}\label{lem:obs1} If $X,Y \in TM$, then 
\begin{align*}
\abs{\norm{X}_g-\norm{Y}_g} \leq d_{TM}(X,Y).
\end{align*}
\end{lemma}

\begin{proof}If $\norm{X}_g = \norm{Y}_g = 0$ then the inequality is trivial. So by relabelling we can assume that $\norm{X}_g \geq \norm{Y}_g$ and $\norm{X}_g \neq 0$. Next let $\sigma:[0,1] \rightarrow TM$ be a geodesic in $TM$ with $\sigma(0)=X$ and $\sigma(1)=Y$.  

First consider the case when $\norm{\sigma(t)}_g \neq 0$ for all $t \in (0,1)$. Then by~\cite[Chapter 2, Corollary 3.3]{dC1992}
\begin{align*}
\frac{d}{dt} \norm{\sigma(t)}_g = \frac{d}{dt} \sqrt{g(\sigma(t), \sigma(t))} = \frac{1}{ \norm{\sigma(t)}_g} g(\sigma(t), \nabla_{\alpha^\prime(t)} \sigma(t))
\end{align*}
where $\alpha = \pi \circ \sigma$. So by Cauchy-Schwarz
\begin{align*}
\abs{\frac{d}{dt} \norm{\sigma(t)}_g} \leq \norm{ \nabla_{\alpha^\prime(t)} \sigma(t)}_g
\end{align*}
Then 
\begin{align*}
d_{TM}(X,Y) 
&= \int_0^1 \sqrt{h(\sigma^\prime(t), \sigma^\prime(t)) dt} = \int_0^1 \sqrt{ \norm{\alpha^\prime(t)}_g^2 +  \norm{ \nabla_{\alpha^\prime(t)} \sigma(t)}_g^2} dt \\ 
& \geq \int_0^1 \norm{ \nabla_{\alpha^\prime(t)} \sigma(t)}_g dt \geq \int_0^1 \abs{\frac{d}{dt} \norm{\sigma(t)}_g }dt\\
& \geq \abs{ \norm{X}_g-\norm{Y}_g}.
\end{align*}

Now consider the case when $\norm{\sigma(t)}_g =0$ for some $t \in (0,1)$. Define 
\begin{align*}
T = \min\{ t \in [0,1] : \norm{\sigma(t)}_g = 0 \}.
\end{align*}
Then $\norm{\sigma(t)}_g \neq 0$ for $t \in (0,T)$ so by the previous argument and using the fact that $\norm{X}_g \geq \norm{Y}_g$ we have
\begin{align*}
d_{TM}(X,Y) \geq d_{TM}(X, \sigma(T)) \geq \abs{\norm{X}_g-\norm{\sigma(T)}_g} = \norm{X}_g  \geq \abs{\norm{X}_g-\norm{Y}_g}.
\end{align*}

\end{proof}

\begin{lemma}\label{lem:obs2} Suppose $0 < \epsilon < 2$, $\gamma: [0,1] \rightarrow TM$ is a smooth path, and 
\begin{align*}
\epsilon \leq \norm{\gamma(t)}_g 
\end{align*}
 for all $t$. If $\sigma: [0,1] \rightarrow SM$ is the curve defined by
\begin{align*}
\sigma(t) = \frac{\gamma(t)}{\norm{\gamma(t)}_g},
\end{align*}
then 
\begin{align*}
\ell_h(\sigma) \leq \frac{2}{\epsilon}\ell_h(\gamma).
\end{align*}
\end{lemma}

\begin{proof} Let $\alpha = \pi \circ \sigma$. By~\cite[Chapter 2, Proposition 2.2]{dC1992} 
\begin{align*}
\nabla_{\alpha^\prime(t)} \sigma(t) = \frac{1}{\norm{\gamma(t)}_g}\nabla_{\alpha^\prime(t)} \gamma(t) + \frac{ \frac{d}{dt} \norm{\gamma(t)}_g}{\norm{\gamma(t)}_g^2} \gamma(t)
\end{align*}
and by~\cite[Chapter 2, Corollary 3.3]{dC1992}
\begin{align*}
\frac{d}{dt} \norm{\gamma(t)}_g = \frac{d}{dt} \sqrt{ g(\gamma(t), \gamma(t))} = \frac{ g( \nabla_{\alpha^\prime(t)} \gamma(t), \gamma(t))}{\norm{\gamma(t)}_g}.
\end{align*}
So by Cauchy-Schwarz
\begin{align*}
\abs{ \frac{d}{dt} \norm{\gamma(t)}_g} \leq \norm{ \nabla_{\alpha^\prime(t)} \gamma(t)}_g
\end{align*}
and so
\begin{align*}
\norm{\nabla_{\alpha^\prime(t)} \sigma(t) }_g \leq \frac{2}{\epsilon} \norm{ \nabla_{\alpha^\prime(t)} \gamma(t)}_g.
\end{align*}

Then 
\begin{align*}
\ell_h(\sigma) 
& = \int_0^1 \sqrt{ h_{\sigma(t)}(\sigma^\prime(t), \sigma^\prime(t))} dt= \int_0^1 \sqrt{ \norm{\alpha^\prime(t)}_g^2+ \norm{\nabla_{\alpha^\prime(t)} \sigma(t)}_g^2} dt \\
& \leq \frac{2}{\epsilon} \int_0^1\sqrt{ \norm{\alpha^\prime(t)}_g^2+ \norm{\nabla_{\alpha^\prime(t)} \gamma(t)}_g^2} dt  = \frac{2}{\epsilon} \ell_h(\gamma). 
\end{align*}
\end{proof}

\begin{lemma}\label{lem:obs3} If $X,Y \in T^1 M$, then 
\begin{align*}
d_{T^1M}(X,Y) \leq d_M(\pi(X), \pi(Y)) + \pi \leq d_{T M}(X,Y) + \pi. 
\end{align*}
\end{lemma}

\begin{proof}
Let $\gamma : [0,T] \rightarrow M$ be a unit speed geodesic joining $\pi(X)$ to $\pi(Y)$. Then let $P(t)$ be the parallel transport of $X$ along $\gamma$. Then, by the definition of the Sasaki metric,
\begin{align*}
d_M(\pi(X), \pi(Y))= \ell_g(\gamma) = \ell_h(P).
\end{align*}
Further, $\norm{P(T)}_g = \norm{X}_g = 1$ and so 
\begin{align*}
d_{T^1M}(P(T), Y) \leq  \pi.
\end{align*}
Thus 
 \begin{align*}
d_{T^1M}(X,Y) \leq d_M(\pi(X), \pi(Y)) + \pi.
\end{align*}

Next let $\alpha : [0,S] \rightarrow TM$ be a geodesic joining $X$ to $Y$. Then, by the definition of the Sasaki metric,
\begin{align*}
d_{TM}(X,Y)= \ell_h(\alpha) \geq \ell_g(\pi \circ \alpha) \geq d_M(\pi(X), \pi(Y)).
\end{align*}
\end{proof}

\begin{proof}[Proof of Proposition~\ref{prop:unit_tangent_vs_tangent}]
By Lemma~\ref{lem:obs3}
\begin{align*}
d_{T^1M}(X,Y) \leq d_M(\pi(X), \pi(Y)) + \pi \leq d_{T M}(X,Y) + \pi. 
\end{align*}
So if $d_{TM}(X,Y) \geq 1$, then 
\begin{align*}
d_{T^1M}(X,Y) \leq d_M(\pi(X), \pi(Y)) + \pi \leq (1+\pi)d_{T M}(X,Y).
\end{align*}

Suppose that $d_{TM}(X,Y) \leq 1$. Let $\gamma : [0,T] \rightarrow TM$ be a unit speed geodesic with $\gamma(0)=X$ and $\gamma(T) =Y$. By Lemma~\ref{lem:obs1}, we must have $\norm{\gamma(t)}_g \geq 1/2$ for all $t$. Let $\sigma(t) = \gamma(t)/\norm{\gamma(t)}_g$. Then by Lemma~\ref{lem:obs2} we have
\begin{align*}
d_{T^1M}(X,Y) \leq \ell_h(\sigma) \leq 4 \ell_h(\gamma) = 4d_{TM}(X,Y).
\end{align*}
\end{proof}

\subsection{Proof of Proposition~\ref{prop:geod_spread}}

To prove the proposition we estimate the growth rate of Jacobi fields. For the rest of the subsection, suppose that $(M,g)$ is a complete Riemannian manfiold with sectional curvature bounded in absolute value by $\kappa > 0$. 

Let $\gamma: \Rb \rightarrow M$ be a geodesic. Let $R$ denote the curvature tensor of $M$. Then for $t \in \Rb$, let $R_{\gamma(t)} : T_{\gamma(t)} M \rightarrow T_{\gamma(t)} M$ denote the map 
\begin{align*}
R_{\gamma(t)}Y = R(\gamma^\prime(t), Y) \gamma^\prime(t).
\end{align*}
Then $R_{\gamma(t)}$ is linear and symmetric relative to $g_{\gamma(t)}$. 

A vector field $J$ along $\gamma$ is called a \emph{Jacobi field} when 
\begin{align*}
\nabla_{\gamma^\prime(t)} \nabla_{\gamma^\prime(t)} J(t) + R_{\gamma(t)}  J(t)=0
\end{align*}
for all $t$. 

We will bound the growth rate of a Jacobi field:

\begin{proposition} If $J$ is a Jacobi field along a geodesic $\gamma$, then 
\begin{align*}
\sqrt{\norm{J(t)}_g^2 + \norm{\nabla_{\gamma^\prime(t)}  J(t)}_g^2} \leq \sqrt{\norm{J(0)}_g^2 + \norm{\nabla_{\gamma^\prime(t)}  J(0)}_g^2} \exp\left( \frac{\kappa+1}{2} t\right)
\end{align*}
for all $t \geq 0$. 
\end{proposition}

\begin{proof} We begin by bounding the operator norm of $R_{\gamma(t)}$. Let $X \in T_{\gamma(t)} M$. Then we can write
\begin{align*}
X = a \gamma^\prime(t) + b Y
\end{align*}
where $a,b \in \Rb$, $Y$ is a unit vector, and $\gamma^\prime(t), Y$ are orthogonal. Then 
\begin{align*}
 g(R_{\gamma(t)}X,X) = b^2 g(R_{\gamma(t)} Y, Y) =  b^2 { \rm sec}(Y, \gamma^\prime(t)) 
 \end{align*}
 since $R_{\gamma(t)}$ is symmetric and $R_{\gamma(t)}\gamma^\prime(t)=0$. Thus
\begin{align*}
\norm{R_{\gamma(t)}X}_g \leq \kappa \norm{X}_g
\end{align*}
for all $X \in T_{\gamma(t)} M$.

Next define $f:\Rb \rightarrow \Rb$ by 
\begin{align*}
f(t) = \sqrt{\norm{J(t)}_g^2 + \norm{\nabla_{\gamma^\prime(t)}  J(t)}_g^2}.
\end{align*}
Then 
\begin{align*}
\abs{f^\prime(t)}
& = \abs{\frac{ \frac{d}{dt} \norm{J(t)}_g^2 + \frac{d}{dt}\norm{\nabla_{\gamma^\prime(t)}  J(t)}_g^2}{2f(t)}}\\
& = \abs{\frac{ \ip{J(t), \nabla_{\gamma^\prime(t)}  J(t)} +\ip{\nabla_{\gamma^\prime(t)}J(t), \nabla_{\gamma^\prime(t)}  \nabla_{\gamma^\prime(t)}J(t)}}{f(t)}}\\
& = \abs{\frac{ \ip{J(t)-R_{\gamma(t)}J(t),\nabla_{\gamma^\prime(t)}J(t) }}{f(t)}} \\
& \leq (\kappa+1)\frac{ \norm{J(t)}_g \norm{\nabla_{\gamma^\prime(t)}J(t)}_g}{f(t)} \\
& \leq \frac{\kappa+1}{2} f(t).
\end{align*}
Then by Gromwall's inequality 
\begin{align*}
f(t) \leq f(0)\exp\left( \frac{\kappa+1}{2}t \right)
\end{align*}
for all $t \geq 0$.
\end{proof}

\begin{proof}[Proof of Proposition~\ref{prop:geod_spread}] Let $\sigma: [0,T] \rightarrow T^1M$ be a unit speed geodesic with $\sigma(0)=\gamma_1^\prime(0)$ and $\sigma(T) = \gamma_2^\prime(0)$. Then consider the map 
\begin{align*}
F: [0,T] \times [0,\infty) \rightarrow T^1M
\end{align*}
given by $F(s,t) = g_t(\sigma(s))$. 

With the decomposition of $T_XTM$ into horizontal and vertical subspaces we then have 
\begin{align*}
\frac{d}{ds} F(s,t) = ( J_s(t), \nabla_{\gamma_s^\prime(t)} J_s(t))
\end{align*}
where $t \rightarrow J_s(t)$ is a Jacobi field along the geodesic $\gamma_s(t) =g_t(\sigma(s))$, see for instance~\cite[Lemma 1.40]{P1999}.

Then 
\begin{align*}
d_M(\gamma_1(t), \gamma_2(t)) 
& \leq d_{T^1M}(\gamma_1^\prime(t), \gamma_2^\prime(t)) \leq \int_{0}^T \norm{\frac{d}{ds} F(s,t)} dt \\
& \leq \exp\left( \frac{\kappa+1}{2}\abs{t} \right)\int_0^T \norm{\frac{d}{ds} F(s,0)}_h  ds \\
&  = \exp\left( \frac{\kappa+1}{2}\abs{t} \right) \int_0^T \norm{\sigma^\prime(s)}_h  ds \\
& = \exp\left( \frac{\kappa+1}{2}\abs{t} \right)  d_{T^1M}(\gamma_1^\prime(0), \gamma_2^\prime(0)).
\end{align*}
\end{proof}

\bibliographystyle{alpha}
\bibliography{complex_kob}

\begin{thebibliography}{DFW14}

\bibitem[Aba89]{A1989}
Marco Abate.
\newblock {\em Iteration theory of holomorphic maps on taut manifolds}.
\newblock Research and Lecture Notes in Mathematics. Complex Analysis and
  Geometry. Mediterranean Press, Rende, 1989.

\bibitem[AT02]{AT2002}
Marco Abate and Roberto Tauraso.
\newblock The {L}indel\"{o}f principle and angular derivatives in convex
  domains of finite type.
\newblock {\em J. Aust. Math. Soc.}, 73(2):221--250, 2002.

\bibitem[Bar80]{B1980}
Theodore~J. Barth.
\newblock Convex domains and {K}obayashi hyperbolicity.
\newblock {\em Proc. Amer. Math. Soc.}, 79(4):556--558, 1980.

\bibitem[BBC14]{BC2014}
Florian Bertrand and L\'ea Blanc-Centi.
\newblock Stationary holomorphic discs and finite jet determination problems.
\newblock {\em Math. Ann.}, 358(1-2):477--509, 2014.

\bibitem[BER00]{BER2000}
M.~S. Baouendi, P.~Ebenfelt, and Linda~Preiss Rothschild.
\newblock Convergence and finite determination of formal {CR} mappings.
\newblock {\em J. Amer. Math. Soc.}, 13(4):697--723, 2000.

\bibitem[Ber03]{berger2003}
Marcel Berger.
\newblock {\em A panoramic view of {R}iemannian geometry}.
\newblock Springer-Verlag, Berlin, 2003.

\bibitem[Bha16]{B2016}
Gautam Bharali.
\newblock Complex geodesics, their boundary regularity, and a
  {H}ardy-{L}ittlewood-type lemma.
\newblock {\em Ann. Acad. Sci. Fenn. Math.}, 41(1):253--263, 2016.

\bibitem[BK94]{BK1994}
Daniel~M. Burns and Steven~G. Krantz.
\newblock Rigidity of holomorphic mappings and a new {S}chwarz lemma at the
  boundary.
\newblock {\em J. Amer. Math. Soc.}, 7(3):661--676, 1994.

\bibitem[BL80]{BL1980}
Steve Bell and Ewa Ligocka.
\newblock A simplification and extension of {F}efferman's theorem on
  biholomorphic mappings.
\newblock {\em Invent. Math.}, 57(3):283--289, 1980.

\bibitem[Bol08]{B2008}
Vladimir Bolotnikov.
\newblock A uniqueness result on boundary interpolation.
\newblock {\em Proc. Amer. Math. Soc.}, 136(5):1705--1715, 2008.

\bibitem[BS92]{BS1992}
Harold~P. Boas and Emil~J. Straube.
\newblock On equality of line type and variety type of real hypersurfaces in
  {${\bf C}^n$}.
\newblock {\em J. Geom. Anal.}, 2(2):95--98, 1992.

\bibitem[BZZ06]{BZZ2006}
Luca Baracco, Dmitri Zaitsev, and Giuseppe Zampieri.
\newblock A {B}urns-{K}rantz type theorem for domains with corners.
\newblock {\em Math. Ann.}, 336(3):491--504, 2006.

\bibitem[CGT82]{CGT1982}
Jeff Cheeger, Mikhail Gromov, and Michael Taylor.
\newblock Finite propagation speed, kernel estimates for functions of the
  {L}aplace operator, and the geometry of complete {R}iemannian manifolds.
\newblock {\em J. Differential Geom.}, 17(1):15--53, 1982.

\bibitem[Che01]{C2001}
Dov Chelst.
\newblock A generalized {S}chwarz lemma at the boundary.
\newblock {\em Proc. Amer. Math. Soc.}, 129(11):3275--3278, 2001.

\bibitem[CL66]{CL1966}
E.~F. Collingwood and A.~J. Lohwater.
\newblock {\em The theory of cluster sets}.
\newblock Cambridge Tracts in Mathematics and Mathematical Physics, No. 56.
  Cambridge University Press, Cambridge, 1966.

\bibitem[CS18]{CS2018}
A.~{Christodoulou} and I.~{Short}.
\newblock {A hyperbolic-distance inequality for holomorphic maps}.
\newblock {\em ArXiv e-prints}, August 2018.

\bibitem[CY80]{CY1980}
Shiu~Yuen Cheng and Shing~Tung Yau.
\newblock On the existence of a complete {K}\"ahler metric on noncompact
  complex manifolds and the regularity of {F}efferman's equation.
\newblock {\em Comm. Pure Appl. Math.}, 33(4):507--544, 1980.

\bibitem[CZ06]{CZ2006}
Bing-Long Chen and Xi-Ping Zhu.
\newblock Uniqueness of the {R}icci flow on complete noncompact manifolds.
\newblock {\em J. Differential Geom.}, 74(1):119--154, 2006.

\bibitem[dC92]{dC1992}
Manfredo Perdig\~ao do~Carmo.
\newblock {\em Riemannian geometry}.
\newblock Mathematics: Theory \& Applications. Birkh\"auser Boston, Inc.,
  Boston, MA, 1992.
\newblock Translated from the second Portuguese edition by Francis Flaherty.

\bibitem[DFW14]{DFW2014}
K.~Diederich, J.~E. Forn{\ae}ss, and E.~F. Wold.
\newblock Exposing points on the boundary of a strictly pseudoconvex or a
  locally convexifiable domain of finite 1-type.
\newblock {\em J. Geom. Anal.}, 24(4):2124--2134, 2014.

\bibitem[DGZ16]{DGZ2016}
Fusheng Deng, Qi'an Guan, and Liyou Zhang.
\newblock Properties of squeezing functions and global transformations of
  bounded domains.
\newblock {\em Trans. Amer. Math. Soc.}, 368(4):2679--2696, 2016.

\bibitem[Eic91]{E1991}
J\"urgen Eichhorn.
\newblock The boundedness of connection coefficients and their derivatives.
\newblock {\em Math. Nachr.}, 152:145--158, 1991.

\bibitem[ELZ03]{ELZ2003}
P.~Ebenfelt, B.~Lamel, and D.~Zaitsev.
\newblock Finite jet determination of local analytic {CR} automorphisms and
  their parametrization by 2-jets in the finite type case.
\newblock {\em Geom. Funct. Anal.}, 13(3):546--573, 2003.

\bibitem[FR18]{FR2018}
John~Erik Forn{\ae}ss and Feng Rong.
\newblock Estimate of the squeezing function for a class of bounded domains.
\newblock {\em Math. Ann.}, 371(3-4):1087--1094, 2018.

\bibitem[Fra91]{F1991}
Sidney Frankel.
\newblock Applications of affine geometry to geometric function theory in
  several complex variables. {I}. {C}onvergent rescalings and intrinsic
  quasi-isometric structure.
\newblock In {\em Several complex variables and complex geometry, {P}art 2
  ({S}anta {C}ruz, {CA}, 1989)}, volume~52 of {\em Proc. Sympos. Pure Math.},
  pages 183--208. Amer. Math. Soc., Providence, RI, 1991.

\bibitem[Hua93]{H1993}
Xiao~Jun Huang.
\newblock Some applications of {B}ell's theorem to weakly pseudoconvex domains.
\newblock {\em Pacific J. Math.}, 158(2):305--315, 1993.

\bibitem[Hua95]{Huang1995}
Xiao~Jun Huang.
\newblock A boundary rigidity problem for holomorphic mappings on some weakly
  pseudoconvex domains.
\newblock {\em Canad. J. Math.}, 47(2):405--420, 1995.

\bibitem[Juh09]{J2009}
Robert Juhlin.
\newblock Determination of formal {CR} mappings by a finite jet.
\newblock {\em Adv. Math.}, 222(5):1611--1648, 2009.

\bibitem[Kap05]{Kap2005}
Vitali Kapovitch.
\newblock Curvature bounds via {R}icci smoothing.
\newblock {\em Illinois J. Math.}, 49(1):259--263, 2005.

\bibitem[Kob05]{K2005}
Shoshichi Kobayashi.
\newblock {\em {Hyperbolic manifolds and holomorphic mappings. An introduction.
  2nd ed.}}
\newblock 2005.

\bibitem[Kra11]{K2011}
Steven~G. Krantz.
\newblock The {S}chwarz lemma at the boundary.
\newblock {\em Complex Var. Elliptic Equ.}, 56(5):455--468, 2011.

\bibitem[KZ16]{KZ2016}
Kang-Tae Kim and Liyou Zhang.
\newblock On the uniform squeezing property of bounded convex domains in
  {$\Bbb{C}^n$}.
\newblock {\em Pacific J. Math.}, 282(2):341--358, 2016.

\bibitem[Lem81]{L1981}
L{\'a}szl{\'o} Lempert.
\newblock La m\'etrique de {K}obayashi et la repr\'esentation des domaines sur
  la boule.
\newblock {\em Bull. Soc. Math. France}, 109(4):427--474, 1981.

\bibitem[Lem82]{L1982}
L.~Lempert.
\newblock Holomorphic retracts and intrinsic metrics in convex domains.
\newblock {\em Anal. Math.}, 8(4):257--261, 1982.

\bibitem[Lem84]{L1984}
L{\'a}szl{\'o} Lempert.
\newblock Intrinsic distances and holomorphic retracts.
\newblock In {\em Complex analysis and applications '81 ({V}arna, 1981)}, pages
  341--364. Publ. House Bulgar. Acad. Sci., Sofia, 1984.

\bibitem[LM07a]{LM2007b}
Bernhard Lamel and Nordine Mir.
\newblock Finite jet determination of {CR} mappings.
\newblock {\em Adv. Math.}, 216(1):153--177, 2007.

\bibitem[LM07b]{LM2007}
Bernhard Lamel and Nordine Mir.
\newblock Parametrization of local {CR} automorphisms by finite jets and
  applications.
\newblock {\em J. Amer. Math. Soc.}, 20(2):519--572, 2007.

\bibitem[LSY04]{LSY2004a}
Kefeng Liu, Xiaofeng Sun, and Shing-Tung Yau.
\newblock Canonical metrics on the moduli space of {R}iemann surfaces. {I}.
\newblock {\em J. Differential Geom.}, 68(3):571--637, 2004.

\bibitem[LSY05]{LSY2004b}
Kefeng Liu, Xiaofeng Sun, and Shing-Tung Yau.
\newblock Canonical metrics on the moduli space of {R}iemann surfaces. {II}.
\newblock {\em J. Differential Geom.}, 69(1):163--216, 2005.

\bibitem[LT16]{LT2016}
Taishun Liu and Xiaomin Tang.
\newblock Schwarz lemma at the boundary of strongly pseudoconvex domain in
  {$\Bbb{C}^n$}.
\newblock {\em Math. Ann.}, 366(1-2):655--666, 2016.

\bibitem[McN89]{McNeal1989}
Jeffery~D. McNeal.
\newblock Holomorphic sectional curvature of some pseudoconvex domains.
\newblock {\em Proc. Amer. Math. Soc.}, 107(1):113--117, 1989.

\bibitem[McN92]{M1992}
Jeffery~D. McNeal.
\newblock Convex domains of finite type.
\newblock {\em J. Funct. Anal.}, 108(2):361--373, 1992.

\bibitem[MY83]{MY1983}
Ngaiming Mok and Shing-Tung Yau.
\newblock Completeness of the {K}\"ahler-{E}instein metric on bounded domains
  and the characterization of domains of holomorphy by curvature conditions.
\newblock In {\em The mathematical heritage of {H}enri {P}oincar\'e, {P}art 1
  ({B}loomington, {I}nd., 1980)}, volume~39 of {\em Proc. Sympos. Pure Math.},
  pages 41--59. Amer. Math. Soc., Providence, RI, 1983.

\bibitem[NA17]{NA2017}
N.~Nikolov and L.~Andreev.
\newblock Boundary behavior of the squeezing functions of {$\Bbb{C}$}-convex
  domains and plane domains.
\newblock {\em Internat. J. Math.}, 28(5):1750031, 5, 2017.

\bibitem[Oss00]{O2000}
Robert Osserman.
\newblock A sharp {S}chwarz inequality on the boundary.
\newblock {\em Proc. Amer. Math. Soc.}, 128(12):3513--3517, 2000.

\bibitem[Pat99]{P1999}
Gabriel~P. Paternain.
\newblock {\em Geodesic flows}, volume 180 of {\em Progress in Mathematics}.
\newblock Birkh\"auser Boston, Inc., Boston, MA, 1999.

\bibitem[Roy71]{R1971}
H.~L. Royden.
\newblock Remarks on the {K}obayashi metric.
\newblock In {\em Several complex variables, {II} ({P}roc. {I}nternat. {C}onf.,
  {U}niv. {M}aryland, {C}ollege {P}ark, {M}d., 1970)}, pages 125--137. Lecture
  Notes in Math., Vol. 185. Springer, Berlin, 1971.

\bibitem[Shi89]{Shi1989}
Wan-Xiong Shi.
\newblock Deforming the metric on complete {R}iemannian manifolds.
\newblock {\em J. Differential Geom.}, 30(1):223--301, 1989.

\bibitem[TLZ17]{TLZ2017}
Xiaomin Tang, Taishun Liu, and Wenjun Zhang.
\newblock Schwarz lemma at the boundary and rigidity property for holomorphic
  mappings on the unit ball of {$\Bbb{C}^n$}.
\newblock {\em Proc. Amer. Math. Soc.}, 145(4):1709--1716, 2017.

\bibitem[Ven89]{Ven1989}
Sergio Venturini.
\newblock Pseudodistances and pseudometrics on real and complex manifolds.
\newblock {\em Ann. Mat. Pura Appl. (4)}, 154:385--402, 1989.

\bibitem[WY17]{WY2017}
D.~{Wu} and S.-T. {Yau}.
\newblock {Invariant metrics on negatively pinched complete K{\"a}hler
  manifolds}.
\newblock {\em ArXiv e-prints}, November 2017.

\bibitem[Yau78]{Y1978}
Shing~Tung Yau.
\newblock A general {S}chwarz lemma for {K}\"ahler manifolds.
\newblock {\em Amer. J. Math.}, 100(1):197--203, 1978.

\bibitem[Yeu09]{Y2009}
Sai-Kee Yeung.
\newblock Geometry of domains with the uniform squeezing property.
\newblock {\em Adv. Math.}, 221(2):547--569, 2009.

\bibitem[Zim17]{Z2017}
Andrew~M. Zimmer.
\newblock Characterizing domains by the limit set of their automorphism group.
\newblock {\em Adv. Math.}, 308:438--482, 2017.

\end{thebibliography}

\end{document}